\documentclass[reqno, 11pt]{amsart}

\usepackage{url}
\usepackage[colorlinks, linkcolor=blue, citecolor=blue, urlcolor=blue]{hyperref}

\usepackage{times}

\newtheorem{theorem}{Theorem}[section]
\newtheorem{lemma}[theorem]{Lemma}

\newtheorem{proposition}[theorem]{Proposition}
\newtheorem{corollary}[theorem]{Corollary}

\theoremstyle{definition}

\newtheorem{ex}[theorem]{Example}

\newtheorem{remark}[theorem]{Remark}

\numberwithin{equation}{section}

\usepackage{bbm}
\usepackage{euscript}
\usepackage{pb-diagram}
\usepackage{lamsarrow}
\usepackage{pb-lams}
\usepackage{amsmath}
\usepackage{amsthm}

\usepackage{mathrsfs}
\usepackage{amsxtra}
\usepackage{amssymb}
\usepackage{pifont}
\usepackage{amsbsy}

\usepackage{graphicx}
\usepackage{epstopdf}

\newskip\aline \newskip\halfaline
\def\skipaline{\vskip\aline}

\def\qedbox{$\rlap{$\sqcap$}\sqcup$}
\def\qed{\nobreak\hfill\penalty250 \hbox{}\nobreak\hfill\qedbox\skipaline}

\def\proofend{\eqno{\mbox{\qedbox}}}

\newcommand{\one}{\mathbbm{1}}

\newcommand\bC{{\mathbb C}}

\newcommand\bR{{\mathbb R}}

\newcommand{\bT}{{\mathbb T}}

\newcommand\bZ{{\mathbb Z}}

\DeclareMathOperator{\re}{\mathbf{Re}}

\DeclareMathOperator{\tr}{{\rm tr}}
\DeclareMathOperator{\supp}{{\rm supp}}

\DeclareMathOperator{\dist}{dist}

 \DeclareMathOperator{\Hom}{Hom}
\DeclareMathOperator{\Vect}{Vect}

\DeclareMathOperator{\spa}{span}

\DeclareMathOperator{\Cr}{\mathbf{Cr}}
\DeclareMathOperator{\Sym}{Sym}

\DeclareMathOperator{\Hess}{Hess}

\DeclareMathOperator{\var}{\boldsymbol{var}}
\DeclareMathOperator{\SO}{SO}
\DeclareMathOperator{\cov}{\boldsymbol{cov}}
\DeclareMathOperator{\Cov}{\boldsymbol{Cov}}
\DeclareMathOperator{\GOE}{GOE}

\DeclareMathOperator{\ind}{ind}

\newcommand{\be}{{\boldsymbol{e}}}

\newcommand{\ii}{\boldsymbol{i}}

\newcommand{\bp}{{\boldsymbol{p}}}
\newcommand{\bq}{{\boldsymbol{q}}}

\newcommand{\bu}{{\boldsymbol{u}}}
\newcommand{\tu}{\check{{\boldsymbol{u}}}}
\newcommand{\bv}{{\boldsymbol{v}}}

\newcommand{\bx}{{\boldsymbol{x}}}

\newcommand{\bsD}{\boldsymbol{D}}
\newcommand{\bsE}{\boldsymbol{E}}

\newcommand{\bsI}{\boldsymbol{I}}

\newcommand{\bsL}{\boldsymbol{L}}

\newcommand{\bsN}{\boldsymbol{N}}

\newcommand{\bsS}{\boldsymbol{S}}

\newcommand{\bsU}{{\boldsymbol{U}}}
\newcommand{\bsV}{\boldsymbol{V}}

\newcommand{\bsZ}{\boldsymbol{Z}}

\newcommand{\bgamma}{\boldsymbol{\gamma}}
\newcommand{\bGamma}{\boldsymbol{\Gamma}}

\newcommand{\bmu}{\boldsymbol{\mu}}

\newcommand{\bom}{{\boldsymbol{\omega}}}
\newcommand{\bsi}{\boldsymbol{\sigma}}
\newcommand{\bSi}{{\boldsymbol{\Sigma}}}

\newcommand{\bXi}{\boldsymbol{\Xi}}

\newcommand{\tsi}{\check{\boldsymbol{\sigma}}}
\newcommand{\cs}{\check{s}}

\newcommand{\si}{{\sigma}}

\newcommand{\ve}{{\varepsilon}}

\newcommand{\vfi}{{\varphi}}

\newcommand{\eE}{\mathscr{E}}

\newcommand{\eH}{\EuScript H}

\newcommand{\eO}{\EuScript{O}}
\newcommand{\eP}{\EuScript{P}}

\newcommand{\eR}{\EuScript{R}}
\newcommand{\eS}{\EuScript{S}}
\newcommand{\eT}{\EuScript{T}}

\newcommand{\eZ}{\EuScript{Z}}

\newcommand{\ra}{\rightarrow}

\newcommand{\Llra}{{\Longleftrightarrow}}

\newcommand{\lan}{\langle}
\newcommand{\ran}{\rangle}

\def\inpr{\mathbin{\hbox to 6pt{\vrule height0.4pt width5pt depth0pt \kern-.4pt \vrule height6pt width0.4pt depth0pt\hss}}}

\newcommand{\pa}{\partial}

\newcommand{\BXi}{\boldsymbol{\Xi}}

\newcommand{\dual}{\spcheck{}}

\begin{document}

\title[Random Morse functions]{Random  Morse functions  and spectral geometry} 



\author{Liviu I. Nicolaescu}
\thanks{This work was partially supported by the NSF grant, DMS-1005745.}

\address{Department of Mathematics, University of Notre Dame, Notre Dame, IN 46556-4618.}
\email{nicolaescu.1@nd.edu}
\urladdr{\url{http://www.nd.edu/~lnicolae/}}


\begin{abstract}   We study random  Morse functions on a Riemann manifold $(M^m,g)$ defined   as a random Gaussian weighted superpositions of eigenfunctions  of the Laplacian of the metric $g$. The randomness is determined by a fixed  Schwartz function $w$ and a small parameter $\ve>0$.  We  first  prove that as $\ve\to 0$  the expected distribution of critical values  of this random function  approaches a   universal measure on $\bR$, independent of $g$,  that can be explictly described  in terms  the expected distribution of  eigenvalues of the Gaussian Wigner  ensemble of  random $(m+1)\times (m+1)$ symmetric matrices.  In contrast, we prove that the metric  $g$  and its curvature  are   determined by  the statistics of the  Hessians  of the random function  for  small $\varepsilon$. \end{abstract}

\maketitle

\tableofcontents

\section{Overview}
\setcounter{equation}{0}

\subsection{The setup}  Suppose that $(M,g)$ is a smooth, compact, connected  Riemann manifold of dimension $m>1$.   We denote by $|dV_g|$  the volume density     on $M$ induced by $g$. We assume that the metric is normalized so that
\[
{\rm vol}_g(M)=1.
\tag{$\ast$}
\label{tag: ast}
\]
For any $\bu, \bv\in C^\infty(M)$ we  denote by $(\bu,\bv)_g$ their $L^2$ inner product defined by  the metric $g$. The $L^2$-norm  of a smooth function $\bu$ is  denoted by $\|\bu\|$.

Let $\Delta_g: C^\infty(M)\ra C^\infty(M)$ denote the scalar Laplacian defined by the metric $g$. Fix  an orthonormal  Hilbert basis $(\Psi_k)_{k\geq 0}$ of $L^2(M)$ consisting of eigenfunctions of $\Delta_g$, 
\[
\Delta_g \Psi_k =\lambda_k \Psi_k,\;\;\|\Psi_k\|=1,\;\; k_0 < k_1 \Rightarrow \lambda_{k_0} \leq \lambda_{k_1}.
\]
Fix  an even measurable functions function $w:\bR\to[0,\infty)$   such that
\[
\lim_{t\to\infty} t^n w(t)=0,\;\;\forall n\in\bZ_{>0}.
\]
  For $\ve>0$  and $k\geq 0$ we set
  \begin{equation}
w_\ve(t):=w(\ve t),\;\;\forall t\in\bR,\;\;v_k^\ve=w_\ve\bigl(\,\sqrt{\lambda_k}\,\bigr).
\label{eq: uve1}
\end{equation}
Consider   random functions  on $M$ of the form
\begin{equation}
\bu^\ve =\sum_{k\geq 0} X_k\sqrt{v_k^\ve}\Psi_k,
\label{eq: uve}
\end{equation}
where the coefficients $X_k$ are independent    standard Gaussian random variables. Note that 
\[
\Delta^N \bu^\ve =\sum_{k\geq 0} \lambda_k ^Nu_k\Psi_k,\;\;\forall N>0.
\]
The fast decay of $w$, the Weyl asymptotic formula, \cite[VI.4]{Cha},  coupled with the Borel-Cantelli lemma imply that  for any  $N>0$ the  function $\Delta^N u^\ve$ is almost surely (a.s.) in $L^2$. In particular, this shows that $\bu_\ve$ is  a.s.  smooth.

The covariance kernel of the Gaussian random function $\bu^\ve$ is given by the function
\[
\eE^\ve: M\times M\to \bR,\;\;\eE^\ve(\bp,\bq)= \bsE\bigl(\,\bu^\ve(p)\bu^\ve(q)\,\bigr)=\sum_{k\geq 0}w_\ve\bigl(\,\sqrt{\lambda_k}\,\bigr) \Psi_k(\bp)\Psi_k(\bq).
\]
The eigenfunctions  $\Psi_k$ satisfy the known pointwise  estimates (see \cite[Thm. 17.5.3]{H3} or \cite[Thm 1.6.1]{N3}),
\[
\|\Psi_k\|_{C^\nu(M)} = O\Bigl(\lambda_k^{\frac{m+\nu}{2}}\Bigr)\;\;\mbox{as $k\to\infty$},\;\;\forall \nu\geq 0.
\]
Since $w_\ve$ is rapidly decresing the above estimates imply  that $\eE^\ve$ is  a smooth function.   More precisely, $\eE^\ve$ is the Schwartz kernel of the smoothing operator
\[
w\bigl(\, \ve\sqrt{\Delta}\,\bigr): C^\infty(M)\to C^\infty(M).
\]
Let us observe that   if $w(0)=1$, then  as $\ve\searrow 0$ the function $w_\ve$ converges uniformly on compacts to the constant function $w_0(t)\equiv 1$  and  $w_\ve(\sqrt{\Delta})$ converges weakly to the  identity operator.     The Schwartz kernel of this limiting operator is the  $\delta$-function on $M\times M$ supported along the diagonal. It defines a  generalized  random function  in the sense of \cite{GeVi2} usually  known  as \emph{white noise}. For this reason, we will refer to the $\ve\to 0$ limits as  \emph{white noise limits}.

In the papers \cite{N2,N21} we investigated  the  distribution of critical points  and critical  values  of the random function $\bu_\ve$ in special case
\[
w(t)=\bsI_{[-1,1]}:=\begin{cases}
1, & |t|\leq 1,\\
0, & |t|>1.
\end{cases}
\]
In this paper we consider  we   investigate   the    same  problem      assuming that $w$ is a \emph{Schwartz function}.  We will  discuss later the  similarities and the  differences between these two situations.

The asymptotic estimates  in Proposition \ref{prop: asymp} show that the  random field $d\bu^\ve$ satisfies the hypotheses of  \cite[Cor. 11.2.2]{AT} for $\ve\ll 1$.   Invoking \cite[Lemma 11.2.11]{AT} we obtain the following   technical result.

\begin{proposition} The random function $\bu^\ve$ is almost surely  Morse if $\ve \ll 1$.\qed
\label{prop: morse}
\end{proposition}

For any $\bu\in C^1(M)$  we  denote by $\Cr(\bu)\subset M $ the set of critical points of $\bu$ and by $\bsD(\bu)$ the set of critical values\footnote{The set $\bsD(\bu)$ is sometime referred to as the \emph{discriminant set} of $\bu$.}  of $\bu$. To a Morse function  $\bu$ on $M$ we  associate  a Borel measure $\mu_\bu$ on $M$ and a Borel measure $\bsi_\bu$ on $\bR$   defined by  the equalities
\[
\mu_\bu: =\sum_{\bp\in \Cr(\bu)}  \delta_\bp,\;\;\bsi_\bu:=\bu_*(\mu_\bu)=\sum_{t\in\bR} |\bu^{-1}(t)\cap\Cr(\bu)|\delta_t.
\]
Observe that
\[
\supp\mu_\bu=\Cr(\bu),\;\;\supp\bsi_\bu=\bsD(\bu).
\]
When $\bu$  is not Morse, we set
\[
\mu_\bu:=|dV_g|,\;\;\bsi_\bu=\delta_0=\mbox{the Dirac  measure on $\bR$ concentrated at the origin}.
\]
Observe that  for any    Morse function $\bu$  and any Borel subset $B\subset \bR$ the number $\bsi_\bu(B)$ is equal to the number of critical values  of $\bu$ in $B$ counted with multiplicity.    We will refer to $\bsi_\bu$  as the \emph{variational complexity} of $\bu$.

To the random function $\bu^\ve$ we associate the random (or empirical) measure  $\bsi_{\bu_\ve}$. Its expectation
\[
\bsi^\ve=\bsE(\si_{\bu^\ve})
\]
is the measure on $\bR$ uniquely determined by the equality
\[
\int_{\bR} f(t) \bsi^\ve(dt)=\bsE\left(\int_\bR f(t) d\bsi_{\bu_\ve}(dt)\right),
\]
for any continuous and bounded function $f:\bR\to\bR$. In \S \ref{s: KR} we show  that  the measure $\bsi^\ve$ is well defined  for $\ve\ll 1$.We will refer to it  as  \emph{the expected  variational complexity} of the random function $\bu_\ve$.        

\begin{enumerate}

\item Describe the  white noise  limit of $\bsi^\ve$.

\item Recover  the   geometry of $(M,g)$ from  white noise statistics of  the random function $\bu_\ve$.

\end{enumerate} 

\begin{remark} Before we state precisely our main results we believe that it is instructive to discuss  some  elementary topologic  and geometric   features of the  white noise behavior of $\bu_\ve$.  For simplicity  we assume  that $w(0)=1$ so that $\bu^\ve$ does converge   to the white noise on $M$. 

It is not hard to prove that for any  given Morse function $f:M\to \bR$ and any $\hbar>0$,  the probability  that $\Vert f-\bu^\ve\Vert_{C^3}<\hbar$ is positive for $\ve$ sufficiently small.   If $f$ happens to be a  stable Morse function, i.e., it has at most one critical point  per level set, then for  $\hbar$ sufficiently small, any   $C^3$-function $g:M\to \bR$ satisfying   $\Vert f-g\Vert_{C^3}<\hbar$ is topologically equivalent to $f$. Thus  as $\ve\to 0$ the random function $\bu_\ve$ samples all the topological types of Morse functions.  

The rescaling $w\to w_\ve$ can be alternatively implemented  as  as follows. Consider the rescaled  metric $g_\ve:=\ve^{-2}g$.    As $\ve\to 0$  the metric  $g_\ve$ becomes flatter and flatter. The  Laplacian of $g_\ve$ is $\Delta_{g_\ve}=\ve^2\Delta_g$. Its eigevalues are $\lambda_k^\ve =\ve^2\lambda_k$  and the collection $\Psi_k^\ve=\ve^{\frac{m}{2}}\Psi_k$ is an orthonormal eigen-basis of $L^2(M, |dV_{g_\ve}|)$.  For any $\ve>0$  we  define the random function
\[
\bv^\ve=\sum_{k\geq 0} X_k w\Bigl(\,\sqrt{\lambda_k^\ve}\,\Bigr)^{\frac{1}{2}}\Psi_k^\ve=\sum_{k\geq 0}  X_k\sqrt{v_k^\ve} \Psi_k^\ve,
\]
where the coefficients $X_k$ are independent  standard Gaussian  random variables.  Observe that $\bv^\ve=\ve^{\frac{m}{2}}\bu^\ve$.    This shows  that the  expected   distribution   $\bsi^\ve(\bv)$ of critical values of $\bv^\ve$ is a rescaling of $\bsi^\ve$. \qed
\end{remark}

\subsection{Statements of the main results} Observe that if $\bu:M\to\bR$ is a  fixed Morse function  and $c$ is a constant,  then 
 \[
 \Cr(c+\bu)=\Cr(\bu), \;\;\mu_{c+\bu}=\mu_\bu, 
 \]
  but   
\[
\bsD(\bu+c)=c+\bsD(\bu),\;\;\bsi_{\bu+c}=\delta_c\ast \bsi_\bu,
\]
where $\ast$ denotes the convolution of two  finite measures on $\bR$.  More generally, if $X$  is a scalar random variable   with  probability distribution  $\nu_X$, then    the  expected variational complexity  of the random function $X+\bu$ is the  measure $\bsE(\,\bsi_{X+\bu}\,)=\nu_X\ast \bsi_\bu$.  If $\bu$ itself is a random function,  and $X$ is independent of $\bu$, then the above equality can be rephrased as
 \[
 \bsE(\,\bsi_{X+\bu}\,)=\nu_X\ast \bsE(\bsi_u).
 \]
 In    particular, if the distribution $\nu_X$ is a Gaussian,  then the  measure   $\bsE(\bsi_\bu)$ is uniquely determined by the  measure  $\bsE(\bsi_{X+\bu})$ since the convolution with a Gaussian is an injective operation.  It turns out that  it is easier to understand the statistics of the variational  complexity  of a  perturbation of $\bu_\ve$  with an independent  Gaussian variable of cleverly chosen variance.
 
 To explain this perturbation we need to introduce several quantities that will play a crucial role throughout this paper.   We define
 \begin{equation}
 \begin{split}
 s_m:=\frac{1}{(2\pi)^m}\int_{\bR^m} w(|x|) dx,\;\;d_m:= \frac{1}{(2\pi)^m}\int_{\bR^m} x_1^2w(|x|) dx,\\
 h_m:= \frac{1}{(2\pi)^m}\int_{\bR^m} x_1^2x_2^2w(|x|) dx.
  \end{split}
 \label{eq: sdh}
 \end{equation}
 The statistical relevance of these quantities is   explained in Proposition \ref{prop: asymp}. If we set
 \begin{equation}
 I_k(w):=\int_0^\infty w(r) r^k dr,
 \label{eq: Ik}
 \end{equation}
 then we deduce  from   \cite[Lemma 9.3.10]{N1}
 \[
 (2\pi)^m s_m=\left(\int_{|x|=1} dA(x)\right) I_{m-1}(w)=\frac{2\pi^{\frac{m}{2}}}{ \Gamma(\frac{m}{2})} I_{m-1}(w),
 \]
 \[
 (2\pi)^m d_m=\left(\int_{|x|=1}x_1^2 dA(x)\right) I_{m+1}(w)= \frac{\pi^{\frac{m}{2}}}{\Gamma(1+\frac{m}{2})} I_{m+1}(w)=\frac{2\pi^{\frac{m}{2}}}{m \Gamma(\frac{m}{2})}I_{m+1}(w),
 \]
 \[
  (2\pi)^m h_m=\left(\int_{|x|=1}x_1^2x_2^2 dA(x)\right) I_{m+1}(w)= \frac{\pi^{\frac{m}{2}}}{2 \Gamma(2+\frac{m}{2})} I_{m+3}(w)=\frac{2\pi^{\frac{m}{2}}}{m(m+2) \Gamma(\frac{m}{2})}I_{m+3}(w).
  \]
  We set 
  \begin{equation}
  q_m:=\frac{s_mh_m}{d_m^2}= \frac{m}{m+2}\frac{I_{m-1}(w) I_{m+3}(w)}{I_{m+1}(w)^2}.
  \label{qm}
  \end{equation}
  The Cauchy inequality implies that $I_{m+1}(w)^2\leq I_{m-1}(w)I_{m+3}(w)$ so that
  \begin{equation}
  q_m\geq \frac{m}{m+2}.
  \label{eq: q1}
  \end{equation}
 The  sequence  $(q_m)_{m\geq 1}$ can be interpreted as a measure of the  tail of $w$, the heavier the tail, the faster the growth of $q_m$ as $m\to \infty$; see  Section \ref{s: 3} for more details.    We set 
  \[
  r_n:=\max(1, q_n),
  \]
   and   define $\bom_m\geq 0$ via the equality
  \begin{equation}
  r_n= \frac{(s_m+\bom_m)h_m}{d_m^2}.
  \label{eq: bom1}
  \end{equation}
  Set  $\cs_m:=s_m+\bom_m$ so that
  \begin{equation}
  r_m=\frac{\cs_mh_m}{d_m^2}.
  \label{eq: sdh1}
  \end{equation}
  Observe that
  \begin{equation}
  \bom_m=0 \Llra  q_m=r_m\geq 1 \Llra \cs_m=s_m,
  \label{eq: bom}
  \end{equation}
  while the inequality (\ref{eq: q1}) implies that
  \begin{equation}
  \lim_{m\to\infty}\frac{\bom_m}{s_m}=0,\;\;\lim_{m\to\infty}\frac{r_m}{q_m}=1.
  \label{eq: rq}
  \end{equation}
  Choose a scalar Gaussian random variable $X_{\bom(\ve)}$ with mean $0$ and variance $\bom(\ve):=\bom_m\ve^{-m}$ independent of $\bu_\ve$ and   form the new random function
  \[
  \tu_\ve:=X_{\bom(\ve)}+ \bu_\ve.
  \]
  We denote by $\tsi^\ve$ the expected  variational complexity  of $\tu_\ve$. We have the  equality
  \begin{equation}
  \tsi^\ve=\bgamma_{\bom(\ve)} \ast \bsi^\ve,\;\;\bom(\ve):=\bom_m\ve^{-m},
  \label{eq: tsi}
  \end{equation}
   Note that
  \[
  \bsN^\ve=\int_\bR \tsi^\ve(dt)=\int_\bR \bsi^\ve(dt)
  \]
  is the  expected number of critical points of the random function $\bu^\ve$.
  
  To formulate our main results we need to briefly recall some  terminology  from random matrix theory.     

For $v\in (0,\infty)$ and $N$ a positive integer we denote by $\GOE_N^v$ the space $\Sym_N$  of real, symmetric  $N\times N$  matrices $A$ equipped with a Gaussian  measure such that the entries $a_{ij}$ are independent, zero-mean,  normal   random variables with variances
\[
\var(a_{ii})=2v,\;\;\var(a_{ij})=v,\;\;\forall 1\leq i <j\leq N.
\]
Let $\rho_{N,v}: \bR\to\bR$ be  the  \emph{normalized correlation} function  of $\GOE^v_N$. It is  uniquely determined  by the equality
\[
\int_{\bR} f(\lambda) \rho_{N,v}(\lambda) d\lambda=\frac{1}{N}\bsE_{\GOE_N^v}\bigl(  \tr f(A)\,\bigr),
\]
for any  continuous  bounded function $f:\bR\ra \bR$.  The function $\rho_{N,v}(\lambda)$  also has a probabilistic interpretation: for any Borel set $B\subset \bR$  the expected number  of eigenvalues in $B$ of  a random $A\in \GOE_N^v$ is equal to
\[
N\int_B\rho_{N,v}(\lambda) d\lambda.
\]
For any $t>0$ we denote by $\eR_t:\bR\ra \bR$ the rescaling map $\bR\ni x\mapsto tx\in \bR$. If $\mu$ is a Borel measure on $\bR$ we denote by $(\eR_t)_*\mu$ its pushforward via the rescaling map $\eR_t$.   The celebrated Wigner semicircle theorem, \cite{AGZ, Me},    states that  as $N\ra \infty$ the rescaled probability measures 
\[
(\eR_{\frac{1}{\sqrt{N}}})_*\bigl(\,\rho_{N,v}(\lambda)d\lambda\,\bigr)
\]
converge  weakly to the  semicircle measure given by the density
\[
\rho_{\infty,v}(\lambda):=\frac{1}{2\pi v}\times \begin{cases}
\sqrt{4v-\lambda^2}, &|\lambda|\leq \sqrt{4v}\\
0, & |\lambda| >\sqrt{4v}.
\end{cases}
\]
We can now state   the  main   results of this paper.

\begin{theorem} For $v>0$  and $N\in\bZ_{>0}$ we set
\[
\theta^\pm_{N,v}(x):= \rho_{N,v}(x)e^{\pm\frac{x^2}{4v}}.
\]
(a) There exists a constant $C=C_m(w)$ that depends only on the dimension $m$ and the weight $w$ such that
\begin{equation}
\bsN^\ve \sim C_m(w) \ve^{-m}\bigl( 1+ O(\ve)\,\bigr)\;\;\mbox{as $\ve \to 0$}.
\label{eq: N}
\end{equation}
More precisely
\begin{equation}
C_m(w) = 2^{\frac{m+4}{2}}r_m^{\frac{1}{2}}\left(\frac{h_m}{2\pi d_m}\right)^{\frac{m}{2}}\Gamma\left(\frac{m+3}{2}\right)\int_\bR (\gamma_{r_m-1}\ast\theta^+_{m+1,r_m})(y)\gamma_1(y) dy.
\label{eq: Cm}
\end{equation}
(b) As $\ve\searrow 0$ the     rescaled  probability measures
\[
\frac{1}{\bsN^\ve}\Bigl(\eR_{\frac{1}{\sqrt{\cs_m\ve^{-m}}}} \Bigr)_*\tsi^\ve
\]
converge  weakly     to a probability  measure   $\tsi_{m}$  on $\bR$ uniquely determined by the proportionalities
\begin{subequations}
\begin{equation}
\tsi_{m}\propto \bigl(\,\gamma_{r_m-1}\ast\theta^+_{m+1,r_m}(y)\,\bigr)\gamma_1(y) dy
\label{eq: lim-si-a}
\end{equation}
\begin{equation}
\propto \theta^-_{m+1,\frac{1}{r_m}}\ast\gamma_{\frac{r_m-1}{r_m}} (y) dy.
\label{eq: lim-si-b}
\end{equation}
\end{subequations}
\label{th: main}
\end{theorem}
One immediate consequence of  Theorem \ref{th: main} is the following universality result. 
\begin{corollary}[\textbf{Universality}]   As $\ve\to 0$   the     rescaled  probability measures
\[
\frac{1}{\bsN^\ve}\Bigl(\eR_{\frac{1}{\sqrt{\cs_m\ve^{-m}}}} \Bigr)_*\bsi^\ve
\]
converge weakly to a probability measure  $\bsi_m$ uniquely determined  by the convolution equation
\[
\gamma_{\frac{\bom_m}{\cs_m}}\ast \bsi_m=\tsi_m.
\]
\label{cor: main}
\end{corollary}
 
Wigner's semicircle theorem \cite[Thm. 2.1.1]{AGZ}   allows us  extract  a bit more about the measure $\bsi_m$ for $m$   large, provided that  the behavior of  $w$   at $\infty$ is not too chaotic. 

\begin{theorem}[\textbf{Central limit theorem}] Suppose that   the weight $w$ is \textbf{regular}, i.e., the sequence $r_m$ has a limit $r\in [1,\infty]$ as $m\to \infty$.  Then
\[
\lim_{m\to\infty}\bsi_m =\gamma_{\frac{r+1}{r}}.
\]
\label{th: main1}
\end{theorem}

The  above regularity assumption  on $w$ is a constraint  on the behavior of  its tail.   In Section \ref{s: 3} we   describe many classes of regular weights. 

\begin{corollary} As $m\to \infty$ we have
\begin{equation}
\begin{split}
 C_m(w)\sim & \frac{8}{\sqrt{\pi m}}\Gamma\left(\frac{m+3}{2}\right)\left(\frac{h_m}{\pi d_m}\right)^{\frac{m}{2}}\\
 \sim & \frac{8}{\sqrt{\pi m}}\Gamma\left(\frac{m+3}{2}\right)\left(\frac{2I_{m+3}(w)}{\pi (m+2) I_{m+1}(w)}\right)^{\frac{m}{2}}.
 \end{split}
 \label{eq: cw}
 \end{equation}
 \label{cor: main1}
 \end{corollary}

Following \cite[\S 12.2]{AT} we define     symmetric $(0,2)$-tensor $h^\ve$ on $M$ given by
\begin{equation}
h^\ve(X,Y)=\frac{\ve^{m+2}}{d_m} \bsE(X\bu_\ve(\bp), Y\bu_\ve(\bp)\,\bigr),\;\;\forall \bp\in M,\;\;X,Y\in\Vect(M),
\label{eq:  hve0}
\end{equation}
where $Xu$ denotes the derivative of the smooth function $u$ along the vector field $X$.

\begin{theorem}[\textbf{Probabilistic reconstruction of the geometry}]

\smallskip

\noindent (a)  For $\ve>0$  sufficiently small  the tensor $h^\ve$ defines a Riemann metric on $M$.

\noindent (b) For any vector  fields $X,Y$ on $M$  the function $h^\ve(X,Y)$ converges uniformly to $g(X,Y)$ as $\ve\searrow 0$.

\noindent (c) The sectional curvatures on $h^\ve$ converge to the  corresponding sectional curvatures of $g$ as $\ve\searrow 0$. 
\label{th: metric}
\end{theorem}

The  $C^0$-convergence of  $h^\ve$ towards the original metric was observed earlier by S. Zelditch \cite{Zel}.   The main novelty of the above theorem is part (c) which, as detailed below, implies the $C^\infty$ convergence of $h^\ve$ to $g$.  However,    the  qualitative  jump from $C^0$ to $C^\infty$-converges requires   requires  novel input.

The construction of the metrics  $h^\ve$  generalizes the construction  in \cite{BBG}. Note that for any $\ve>0$ we have a smooth  map $\bXi_\ve:M\to L^2(M,g)$
\begin{equation}
M\ni p\mapsto \BXi_\ve(\bp):= \left(\frac{\ve^{m+2}}{d_m}\right)^{\frac{1}{2}}\sum_{k\geq 0}w_\ve\bigl(\,\sqrt{\lambda_k}\,\bigr)^{\frac{1}{2}}\Psi_k(\bp)\Psi_k\in L^2(M,g).
\label{xive}
\end{equation}
Then $h^\ve$ is the pullback by $\bXi_\ve$ of the   Euclidean metric on $L^2(M,g)$.   Let us point out that \cite[Thm.5]{BBG} is a special case of Theorem \ref{th: metric}  corresponding to the weight $w(t)= e^{-t^2}$.  

Theorem \ref{th: metric} coupled with the results in \cite{Pet} imply that the metrics $h^\ve$ converge  $C^{1,\alpha}$ to $g$ as $\ve\searrow 0$. The convergence of  sectional curvatures coupled with the technique of harmonic coordinates in \cite{AC, Pet} can be used to bootstrap this convergence to a $C^\infty$ convergence.

\begin{remark}  Suppose that $w$ has compact support, say $\supp w\subset [-1,1]$. In this case the map $\BXi_\ve$ is actually a map to  the \emph{finite dimensional} Euclidean space
\[
\bsU_\ve:= \spa\bigl\{\Psi_k;\;\;\lambda_k\leq \ve^{-2}\,\bigr\}\subset L^2(M,g).
\]
Theorem \ref{th: metric}  implies that for $\ve>0$ sufficiently small the map $\BXi_\ve$ is a near-isometric embedding of $M$ in a \emph{finite dimensional} space.  It is   conceivable that this near-isometric  embedding could be deformed to an actual isometry  by using the strategy of X. Wang and K. Zhu \cite{WaZhu}.\qed
\end{remark}

We should add a few words about  the  nontrivial  analytic result hiding behind Theorem \ref{th: metric}. Fix a point $\bp\in M$ and normal coordinates $(x^i)$ at $\bp$.      The  techniques pioneered by L. H\"{o}rmander \cite{Hspec} show that  as $\ve\searrow 0$ we have  the $1$-term  asymptotic expansions
\begin{subequations}
\begin{equation}
\bsE\Bigl(\,\pa^2_{x^ix^i}\bu^\ve(\bp)\,\cdot \, \pa^2_{x^jx^j}\bu^\ve(\bp)\,\Bigr) = h_m\ve^{-(m+4)}\Bigl(\,1+O(\ve^2)\,\Bigr),
\label{eq: rough1}
\end{equation}
\begin{equation}
\bsE\Bigl(\,\pa^2_{x^ix^j}\bu^\ve(\bp)\,\cdot \, \pa^2_{x^ix^j}\bu^\ve(\bp)\,\Bigr) = h_m\ve^{-(m+4)}\Bigl(\,1+O(\ve^2)\,\Bigr).
\label{eq: rough2}
\end{equation}
\end{subequations}
These and several other similar $1$-term  asymptotic expansions  involving the Schwartz kernel of $w_\ve(\sqrt{\Delta})$ (see Proposition \ref{prop: asymp})   are  responsible for Theorem \ref{th: main}.     All these $1$-term expansions are \emph{independent} of the background metric $g$.   Note that  (\ref{eq: rough1}) and (\ref{eq: rough2})  imply the estimate
\begin{equation}
\bsE\Bigl(\,\pa^2_{x^ix^i}\bu^\ve(\bp)\,\cdot \, \pa^2_{x^jx^j}\bu^\ve(\bp)\,\Bigr)-\bsE\Bigl(\,\pa^2_{x^ix^j}\bu^\ve(\bp)\,\cdot \, \pa^2_{x^ix^j}\bu^\ve(\bp)\,\Bigr)= O\bigl(\, \ve^{-(m+2)}\,\bigr).
\label{eq: rough}
\end{equation}
Theorem \ref{th: metric} is equivalent with  the   following  sharp estimate
\[
\bsE\Bigl(\,\pa^2_{x^ix^i}\bu^\ve(\bp)\,\cdot \, \pa^2_{x^jx^j}\bu^\ve(\bp)\,\Bigr)-\bsE\Bigl(\,\pa^2_{x^ix^j}\bu^\ve(\bp)\,\cdot \, \pa^2_{x^ix^j}\bu^\ve(\bp)\,\Bigr)\sim d_mK^g_{ij}(\bp)\ve^{-(m+2)},
\]
where $K_{ij}^g(\bp)$ denotes the sectional curvature of $g$ at $\bp$ along the $2$-plane  spanned by  $\pa_{x^i},\pa_{x^j}$ which is an obvious refinement of  (\ref{eq: rough}).  In fact    the proof of Theorem \ref{th: metric} is based on   refinements of   the $1$-term expansions (\ref{eq: rough1}) and (\ref{eq: rough2}) to explicit $2$-term expansions that we prove in  Appendix \ref{s: spec}.

The convergence of  the metrics $h^\ve$ leads to a cute probabilistic proof of the Gauss-Bonnet theorem for the original metric $g$ (and thus for any metric on $M$).  Here is the simple principle behind this proof. 

     Assume  for simplicity that $M$ is oriented and $m=\dim M$ is even.  To a Morse function   $f$  we associate the  signed measure 
\[
\nu_f=\sum_{df(\bp)=0}(-1)^{\ind(f,p)} \delta_\bp,
\]
where $\ind(f,\bp)$ denotes the Morse index  of the critical point of the Morse function $f$.   The Poincar\'{e}-Hopf theorem implies that
\begin{equation}
\int_M \nu_f =\chi(M).
\label{eq: P_H}
\end{equation}
We can also think of  $\nu_f$  as  a degree $0$-current.  The random  function  $\bu^\ve$  then determines   a random $0$-current $\nu_{\bu^\ve}$. It turns out (see Section \ref{s: 4}) that the expectation  of this current is a current represented by a rather canonical top degree  form. More precisely, we prove that, 
\begin{equation}
\bsE\bigl(\, \nu_{\bu^\ve}\,\bigr)= \be_{h^\ve}(M),
\label{eq: prob-GB}
\end{equation}
where  $\be_{h^\ve}(M)$ is  the  Euler form  defined  by the metric $h^\ve$ which appears in the Gauss-Bonnet theorem.   Using (\ref{eq: P_H}) we conclude that 
\[
\int_M\be_{h^\ve}(M)=\int_M\bsE\bigl(\, \nu_{\bu^\ve}\,\bigr)= \bsE\left(\int_M \nu_{\bu^\ve}\right)=\chi(M),
\]
 and as a bonus  we  the Gauss-Bonnet theorem for the metric  $h^\ve$. Letting $\ve\to 0$   we obtain the Gauss-Bonnet theorem for $g$ since $h^\ve\to g$ and $\be_{h^\ve}(M)\to \be_g(M)$.  In particular, this shows that $\bsE(\, \nu_{\bu^\ve}\,)$ converges in the sense of currents to $\be_g(M)$, the Euler form determined by the metric $g$.

\subsection{Some perspective} In \cite{N21} we proved     the counterparts of Theorem \ref{th: main}, Corollary \ref{cor: main} and Theorem  \ref{th: main1} in the case of the singular weight $w=\bsI_{[-1,1]}$.  In this case the random function $\bu^\ve$ could be loosely  interpreted as a  random polynomial  of large degree because  since this is  the case when $(M,g)$ is the round sphere.

The fact   that the results in the  singular case  $w=\bsI_{[-1,1]}$ are very similar to the results in the smooth case when $w$  is Schwartz function could be erroneously  interpreted as an indication   that there are no qualitative differences between these two situations.   This is not the case.

There is one subtle  and meaningful   qualitative difference    burried  in the proofs of Theorem \ref{th: main} and Theorem \ref{th: main1}. It has to do  with the size of the tail of $w$ as encoded by the quality  $q_m=q_m(w)$ defined in (\ref{qm}).  Loosely speaking,  a large $q_m$ is an indication  of a heavy tail.  

The proof  of Theorem \ref{th: main} requires different arguments    depending on whether $q_m>1$ or $q_m\leq 1$; see Case 1 and Case 2 in the proof of Theorem \ref{th: main}. Since $q_m(w)\geq\frac{m}{m+2}$ for any $w$  we see that for $m$ large  the situation $q_m<1$ is rather atypical. The case of the singular  weight $w=\bsI_{[0,1]}$   is atypical because  $q_m(w)=\frac{m+2}{m+4}<1$.

The size of the tail plays  an even more fundamental role  in the proof of the the Central Limit Theorem \ref{th: main1}.       The large $m$-limit of $\bsi_m$  exists because of  two facts: Wigner's semicircle  theorem and the fact limit  $\lim_m q_m=r=r(w)$ exists. However,  the proof  depends heavily on the size of the tail and there are  two dramatically  different cases, $r<\infty$ and $r=\infty$.    The fact  that    the central limit theorem   has a similar statement  in both cases is a bit surprising  because different forces are at play in these two cases.

In Section \ref{s: 3} we show that the  two behaviors, $r<\infty$ and $r=\infty$ are not just theoretically possible,  they can actually happen  for various choices of $w$. For example,  if
\[
w(t)\sim \frac{1}{t^{\log\log(t)}}\;\;\mbox{as}\;\;t\to \infty,
\]
then  $r=\infty$. If $w(t)\sim e^{-c(\log t)^2}$ as $t\to\infty$ for some $c>0$, then $r= e^{8/c}$, while if $w(t)=e^{-t^2}$, then $r=1$.

The quantity  $r(w)$ also affects the size of the constant $C_m(w)$ in (\ref{eq: N}) which  states that the expected number of critical points of $\bu^\ve$ is asymptotic to $C_m(w)\ve^{-m}$ as $\ve\to 0$. For example if  
\[
w(t)\sim \frac{1}{t^{\log\log(t)}}\;\;\mbox{as}\;\;t\to \infty,
\]
then 
\[
\log C_m(w)\stackrel{(\ref{cminfi})}{\sim} \frac{m}{2}e^{m+2}(e^2-1)\;\;\mbox{as}\;\;m\to \infty.
\]
If $w(t)\sim e^{-c(\log t)^2}$ as $t\to\infty$, then
\[
\log C_m(w) \stackrel{(\ref{cm2})}{\sim} \frac{2}{C} m^2 \;\;\mbox{as}\;\;m\to \infty.
\]
If $w(t)=e^{-t^2}$, then 
\[
\log C_m(w)\stackrel{(\ref{eq: logcm})}{\sim} \frac{1}{2}m\log m \;\;\mbox{as}\;\;m\to \infty.
\]
It is well known that if $w$ is a Schwartz function, then  Schwartz kernel of $w_\ve(\sqrt{\Delta})$   has a complete  asymptotic  expansion   as $\ve\searrow 0$; see  e.g. \cite[Chap. XII]{Tay}. Very little  was known about   the precise nature of the expansion. In  Theorem \ref{prop: two-step} and we  obtain an explicit description  of the  second term of this expansion in terms of  geometric invariants of the Riemann manifold $(M,g)$.

 Theorem \ref{prop: two-step}  is  a new result and we  have delegated it to an  appendix not to diminish its importance, but  to  help the reader   separate  the two  conceptually different facts responsible for  Theorem \ref{th: metric}.

The first fact is purely probabilistic   and states that the Riemann tensor of the approximate metric $h_\ve$ is  completely determined by the statistics of the Hessians of   $\bu^\ve$; see (\ref{kijve}) and (\ref{eq: rhess}).

The second fact is purely analytic:  the  asymptotics of the  statistics of the Hessians as $\ve\to 0$   are  determined by  the   $2$-term expansion  proved in Theorem \ref{prop: two-step}.

\subsection{The organization of the paper} The remainder of the paper is organized as follows.  Section \ref{s: pro} contains the proofs of the main  results. In Section \ref{s: 3}  we  describe many classes of regular weights $w$.   In particular, these  examples show that the limit $r=\lim_{m\to \infty} r_m$ that appears in the statement of Theorem \ref{th: main1} can have any value in $[1,\infty]$.   Section \ref{s: 4} contains the details of the  probabilistic proof of the Gauss-Bonnet theorem outlined above.
 
To smooth the flow of the presentation we   gathered in Appendices various technical results used  in the proofs of the mains results.  In Appendix \ref{s: dist}  we describe the  jets of order $\leq 4$  along the diagonal of the  square of the distance function $\dist_g: M\times M\to \bR$ which are needed in the two-step asymptotics of the correlation kernel.  This  feels  like a classical problem, but since precise  references are hard to find we decided to include a complete proof.  Our approach, based on the Hamilton-Jacobi equation satisfied by the distance function is similar to the one sketched  in \cite[p.281-282]{DeW}.

In Appendix \ref{s: spec}   we   describe  small $\ve$ asymptotics of the Schwartz kernel of $w(\ve\sqrt{\Delta})$  using a  strategy   pioneered  L. H\"{o}rmander \cite{Hspec} based on a good understanding of the short time asymptotics  for the wave kernel.        For the   applications in this paper we need \emph{explicit, two-term} asymptotics.  The central  result in this appendix is  Theorem  \ref{prop: two-step} which, to the best of our knowledge,    was never published  before.   It essentially states   that the Riemann curvature tensor can be recovered   from the  second order terms of  the $\ve\to 0$ asymptotics of the fourth order  jets  along the diagonal of the Schwartz kernel  of $w(\ve\sqrt{\Delta})$.

In Appendix \ref{s: gauss} we describe a few facts about Gaussian measures   in a coordinate free form suitable for our geometric purposes. Finally, in   Appendix \ref{s: gmat}  we   have collected   some facts  about a family of Gaussian random symmetric matrices that appear in our investigation and are less familiar to   differential geometers.

\subsection{Notations}

\begin{enumerate}

\item For any set $S$ we denote by $|S| \in \bZ_{\geq 0}\cup\{\infty\}$ its cardinality. For any subset $A$ of a set $S$ we denote by $\bsI_A$ its characteristic function
\[
\bsI_A: S\ra \{0,1\},\;\;\bsI_A(s)=\begin{cases}
1, & s\in A\\
0, &s\in S\setminus A.
\end{cases}.
\]

\item  For any point $x$ in a smooth manifold $X$ we denote by $\delta_x$ the Dirac  measure on $X$ concentrated at $x$.

\item For any smooth manifold $M$    we denote by $\Vect(M)$ the vector space of smooth vector fields on $M$.

\item For any random variable $\xi$ we denote by $\bsE(\xi)$ and respectively $\var(\xi)$ its expectation  and respectively its variance.

\item For any finite dimensional real  vector space $\bsV$ we denote by $\bsV\dual$ its dual, $\bsV\dual:=\Hom(\bsV,\bR)$.


\item For any Euclidean space $\bsV$ we denote by $\Sym(\bsV)$ the space of symmetric  linear operators $\bsV\ra \bsV$. When $\bsV$ is  the Euclidean space $\bR^m$ we set $\Sym_m:=\Sym(\bR^m)$. We denote by $\one_m$ the identity map $\bR^m\to\bR^m$.

\item We denote by $\eS(\bR^m)$ the space of Schwartz functions on $\bR^m$.

\item For $v>0$ we  denote by $\gamma_v$  the  centered Gaussian measure on $\bR$ with variance $v$, 
\[
\gamma_v(x)dx = \frac{1}{\sqrt{2\pi v}}e^{-\frac{x^2}{2v}} |dx|.
\]
Since $\lim_{v\searrow 0}  \gamma_v=\delta_0$, we  set $\gamma_0:=\delta_0$.  For a  real valued random variable  $X$ we write $X\in \bsN(0,v)$ if the probability distribution of $X$ is $\gamma_{v}$.

\item If  $\mu$ and $\nu $ are two  finite measures on a common space $X$, then the notation  $\mu\propto \nu$   means  that 
\[
\frac{1}{\mu(X)}\mu=\frac{1}{\nu(X)}\nu.
\]

\end{enumerate}

\section{Proofs}
 \label{s: pro}
\setcounter{equation}{0} 
 
\subsection{A Kac-Rice type formula}
\label{s: KR}

The key  result  behind   Theorem \ref{th: main} is a Kac-Rice type result which we intend to discuss in some detail in this  section.    This  result gives an explicit, yet  quite complicated description of the measure  $\tsi^\ve$. More precisely, for any Borel subset $B\subset \bR$, the Kac-Rice formula provides an integral representation of $\tsi^\ve(B)$ of the form
\[
\tsi^\ve(B)=\int_M f_{\ve,  B}(\bp)\,|dV_g(\bp)|,
\]
for some integrable function $f_{\ve, B}: M\ra \bR$. The core of the Kac-Rice  formula is an explicit  probabilistic description of the density $f_{\ve,B}$.

Fix a point $\bp\in M$.      This determines three   Gaussian random variables
\begin{equation*}
\tag{$RV$}
\begin{split}
\tu_\ve (\bp)\in\bR,\;\; d\tu_\ve (\bp)\in T^*_\bp M, \;\;\Hess_\bp(\tu_\ve)\in \Sym(T_\bp M),
\end{split}
\label{eq: rvmain}
\end{equation*}
where $\Hess_\bp(\tu_\ve): T_\bp M\times T_\bp M\ra \bR$ is the Hessian of $\bu_\bom$ at $\bp$ defined in terms  of the Levi-Civita  connection of $g$ and then identified with a symmetric endomorphism of $T_\bp M$  using  again the metric $g$.     More concretely, if $(x^i)_{1\leq i\leq m}$ are $g$-normal coordinates at $\bp$, then
\[
\Hess_\bp(\tu_\ve)\pa_{x^j}=\sum_{i=1}^m\pa^2_{x^ix^j}\tu_\ve(\bp)\pa_{x^i}.
\]
For $\ve >0$  sufficiently  small  the covariance form of the  Gaussian random vector $d\tu_\ve(\bp)$  is positive definite; see (\ref{eq: bss}). We can identify it with a  symmetric, positive definite   linear operator
\[
\bsS\bigl(\, d\tu_\ve(\bp)\,\bigr) : T_\bp M\ra T_\bp M.
\]
More  concretely, if $(x^i)_{1\leq i\leq m}$ are  $g$-normal coordinates  at $\bp$, then we  identify $\bsS\bigl(\, d\tu_\ve(\bp)\,\bigr)$ with a $m\times m$ real symmetric  matrix whose $(i,j)$-entry is given by
\[
\bsS_{ij}\bigl(\, d\tu_\ve(\bp)\,\bigr)=\bsE\bigl(\, \pa_{x_i}\tu_\ve(\bp)\cdot \pa_{x^j}\tu_\ve(\bp)\,\bigr).
\]
 \begin{theorem}  Fix a Borel subset $B\subset \bR$. For any $\bp\in M$ define
 \[
 f_{\ve, B}(\bp):=\left(\,\det\bigl(\, 2\pi\bsS(\,\tu_\ve(\bp)\,\,\bigr)\,\right)^{-\frac{1}{2}}\bsE\Bigl(\; |\det \Hess_\bp(\tu_\ve)|\,\cdot\, \bsI_B(\,\tu_\ve(\bp)\,)\;\;\bigl|\; d\tu_\ve(\bp)=0\;\Bigr),
 \]
where  $\bsE\bigl( \;\mathbf{var}\;\; | \; \mathbf{cons}\;\bigr)$ stands for the conditional expectation  of the variable $\mathbf{var}$ given the constraint $\mathbf{cons}$.  Then 
 \begin{equation}
 \tsi^\ve(B)=\int_M f_{\ve, B}(\bp)\,|dV_g(\bp)|.
 \label{eq: KR}
 \end{equation}
 \qed
 \label{th: KR}
 \end{theorem}
 
 This theorem is a special case of a general result of Adler-Taylor,  \cite[Cor. 11.2.2]{AT}.   Proposition \ref{prop: asymp} below shows that the    technical assumptions in  \cite[Cor. 11.2.2]{AT}  are   satisfied if $\ve \ll 1$.

For the above theorem  to be of any use we need to have some concrete  information about the Gaussian  random variables (\ref{eq:  rvmain}).    All the relevant statistical invariants of these  variables can be extracted from the covariance kernel of the random function $\tu_\ve$.

\subsection{Proof of  Theorem \ref{th: main}}
\label{s: main}

Fix $\ve> 0$. For any $\bp\in M$,  we have the centered  Gaussian   random vector  
\[
 \bigl(\,\tu_\ve(\bp), d\tu_\ve(\bp),\Hess_\bp(\tu_\ve)\,\bigr)\in\bR\oplus T_\bp^*M\oplus \Sym(T_\bp M).
\]
We fix normal coordinates $(x^i)_{1\leq i\leq m}$ at $\bp$  and we can identify the above Gaussian vector  with the centered Gaussian vector
\[
\bigl(\,\tu_\ve(\bp),\; (\pa_{x^i}\tu_\ve(p)\,)_{1\leq i\leq m},\;\pa^2_{x^ix^j}(\tu_\ve(\bp)\,)_{1\leq i,j\leq m}\,\bigr)\in\bR\oplus \bR^m\oplus \Sym_m.
\]
The next result is the  key reason the Kac-Rice formula  can  be applied succesfully to the problem at hand.

\begin{proposition}  For any $1\leq i,j,k,\ell\leq m$ we have the  uniform in $\bp$ asymptotic estimates as $\ve\searrow 0$
\begin{subequations}
\begin{equation}
\bsE(\,\tu_\ve(p)^2\,\bigr)= \check{s}_m\ve^{-m}\bigl(1+O(\ve^2)\,\bigr),
\label{eq: cov0}
\end{equation}
\begin{equation}\bsE\bigl(\, \pa_{x^i}\tu_\ve(\bp)\pa_{x^j}\tu_\ve(\bp)\,\bigr)= d_m\ve^{-(m+2)}\delta_{ij}\bigl(\,1+O(\ve^2)\,\bigr),
\label{eq: cov1}
\end{equation}
\begin{equation}
\bsE\bigl(\, \pa^2_{x^ix^j}\tu_\ve(\bp)\pa^2_{x^kx^\ell}\tu_\ve(\bp)\,\bigr)= h_m\ve^{-(m+4)}(\delta_{ij}\delta_{k\ell}+\delta_{ik}\delta_{j\ell}+ \delta_{i\ell}\delta_{jk})\bigl(\,1+O(\ve^2)\,\bigr),
\label{eq: cov2}
\end{equation}
\begin{equation}
\bsE\bigl(\,\tu_\ve(\bp)\pa^2_{x^ix^j}\tu_\ve(\bp)\,\,\bigr)= -d_m\ve^{-(m+2)}\delta_{ij}\bigl(\,1+O(\ve^2)\,\bigr),
\label{eq: cov3}
\end{equation}
\begin{equation}
\bsE\bigl(\,\tu_\ve(p)\pa_{x^i}\tu_\ve(\bp)\,\bigr)= O(\ve^{-m}),\;\; \bsE\bigl(\,\pa_{x^i}\tu_\ve(\bp)\pa^2_{x^jx^j}\tu_\ve(\bp)\,\bigr)=O(\ve^{-(m+2)}),
\label{eq: cov4}
\end{equation}
\end{subequations}
where $\check{s}_m=s_m+\bom_m$ and the constants $s_m, d_m, h_m$ are defined  by (\ref{eq: sdh}). \qed
\label{prop: asymp}
\end{proposition}

\begin{proof} Denote by $\check{\eE}^\ve$ the covariance kernel of the random function 
\[
\tu_\ve= X_{\bom(\ve)}+ \bu_\ve.
\]
Note that  
\[
\check{\eE}^\ve(\bp,\bq)=\bom(\ve)+\eE^\ve(\bp,\bq)=\bom_m\ve^{-m}+ \eE^\ve(\bp,\bq).
\]
Fix a point $\bp_0\in M$ and normal coordinates  at $\bp_0$ defined  in an open neighborhood $\eO_0$ of $\bp_0$. The restriction of $\eE^\ve$ to  $\eO_0\times \eO_0$ can be viewed as a function $\eE^\ve(x,y)$ defined in an open neighborhood of $(0,0)$ in $\bR^m\times \bR^m$.   For any  $\alpha,\beta\in (\bZ_{\geq 0})|^m$  we have
\[
\bsE\bigl(\,\pa^\alpha_x\tu_\ve(\bp_0)\pa^\beta_x(\tu_\ve)\,\bigr)=\pa^\alpha_x \pa^\beta_y \check{\eE}^\ve(x,y)_{x=y=0}.
\]
Proposition \ref{prop: asymp} is now a consequence of    the spectral estimates (\ref{eq: key-asym}) in Appendix \ref{s: spec}.
\end{proof}

From the estimate (\ref{eq: cov1})  we deduce that
\begin{equation}
\bsS(\, d\tu_\ve(\bp)\,) =d_m\ve^{-(m+2)}\bigl(\,\one_m+O(\ve^2)\,\bigr),
\label{eq: bss}
\end{equation}
so that
\begin{equation}
\sqrt{|\det \bsS(\tu_\ve(p))|}= (d_m)^{\frac{m}{2}} \ve^{-\frac{m(m+2)}{2}} \bigl(\,1+O(\ve^2)\,\bigr)\;\;\mbox{as $\ve\to 0$}.
\label{eq: asy-det}
\end{equation}
Consider the rescaled random vector
\[
(\check{s}^\ve, v^\ve, H^\ve)=:= \bigl(\ve^{\frac{m}{2}} \tu_\ve(\bp) , \ve^{\frac{m+2}{2}} d\tu_\ve(p),\;\; \ve^{\frac{m+4}{2}} \nabla^2 \tu_\ve(\bp)\,\bigr).
\]
Form Proposition \ref{prop: asymp} we deduce  the following (uniform in $\bp$) estimates as $\ve\searrow 0$.
\begin{subequations}
\begin{equation}
\bsE(\,(\check{s}^\ve)^2 \,\bigr)= \check{s}_m\bigl(1+O(\ve^2)\,\bigr),
\label{eq: cov0r}
\end{equation}
\begin{equation}\bsE\bigl(\, v_i^\ve v_j^\ve\,\bigr)= d_m\delta_{ij}\bigl(\,1+O(\ve^2)\,\bigr),
\label{eq: cov1r}
\end{equation}
\begin{equation}
\bsE\bigl(\, H^\ve_{ij}H^\ve_{kl}\,\bigr)= h_m(\delta_{ij}\delta_{k\ell}+\delta_{ik}\delta_{j\ell}+ \delta_{i\ell}\delta_{jk})\bigl(\,1+O(\ve^2)\,\bigr),
\label{eq: cov2r}
\end{equation}
\begin{equation}
\bsE\bigl(\,\check{s}^\ve H_{ij}^\ve\,\,\bigr)= -d_m\delta_{ij}\bigl(\,1+O(\ve^2)\,\bigr),
\label{eq: cov3r}
\end{equation}
\begin{equation}
\bsE\bigl(\,\check{s}^\ve v^\ve_i\,\bigr)= O(\ve),\;\; \bsE\bigl(\,v^\ve_iH^\ve_{jk}\,\bigr)=O(\ve).
\label{eq: cov4r}
\end{equation}
\end{subequations}
The probability distribution of the variable $s^\ve$ is 
\[
d\gamma_{\cs_m(\ve)}(x)=\frac{1}{\sqrt{2\pi \cs_m(\ve)}}e^{-\frac{x^2}{2\cs_m(\ve)}} |dx|,
\]
where 
\[
\cs_m(\ve)=\check{s}_m+O(\ve).
\]
Fix a Borel set $B\subset\bR$. We have
\begin{equation}
\begin{split}
\bsE\bigl(\, |\det\nabla^2 \tu_\ve(\bp)| \bsI_{ B}\bigl(\,\tu_\ve(\bp)\,\bigr)\;\bigl| \; d\tu_\ve(\bp)=0\,\bigr)  =\ve^{-\frac{m(m+4)}{2}}\bsE\bigl(\, |\det H^\ve| \bsI_{\ve^{\frac{m}{2}}B}(\cs^\ve)\;\bigl| \; v^\ve=0\,\bigr)& \\
=\ve^{-\frac{m(m+4)}{2}}\underbrace{\int_{\ve^{\frac{m}{2}}B} \bsE\bigl(\, |\det H^\ve| \;\bigl|\; \cs^\ve=x,\;v^\ve=0\,\bigr)  \frac{e^{-\frac{x^2}{2\cs_m(\ve)}}}{\sqrt{2\pi \cs_m(\ve)}} |dx|}_{=:q_{\ve,\bp}(\ve^{\frac{m}{2}}B)}.
\end{split}
\label{eq: asy-exp}
\end{equation}
Using (\ref{eq: asy-det}) and (\ref{eq: asy-exp}) we deduce from Theorem \ref{th: KR} that
\[
\tsi^\ve(B) =\ve^{-m}\left(\frac{1}{2\pi d_m}\right)^{\frac{m}{2}}\int_M q_{\ve,\bp}(\ve^{\frac{m}{2}} B) \rho_L(\bp) |dV_g(\bp)|,
\]
where   $\rho_\ve:M\ra \bR$ is a function that satisfies the uniform in $\bp$ estimate 
\begin{equation}
\rho_\ve(p)=1+O(\ve)\;\;\mbox{as $\ve\to 0$}.
\label{eq: rhol}
\end{equation}
Hence
\begin{equation}
\ve^m\left(\eR_{ \ve^{\frac{m}{2}}}\right)_* \tsi^\ve(B)=  \left(\frac{1}{2\pi d_m}\right)^{\frac{m}{2}}\int_M q_{\ve,\bp}(B) \rho_\ve(\bp) |dV_g(\bp)|.
\label{eq: dlb}
\end{equation}
To continue the  computation we need to investigate the behavior of $q_{\ve,\bp}(B)$ as $\ve$. More concretely,  we need  to elucidate the nature of the Gaussian  vector
\[
\bigl(\, H^\ve\; \bigl|\; \cs^\ve=x,\;v^\ve=0\,\bigr).
\]
We will achieve this via the regression formula (\ref{eq: cov-regr}). For simplicity we set
\[
Y^\ve:= (\cs^\ve, v^\ve)\in\bR\oplus \bR^m.
\]
The components of $Y^\ve$ are   
\[
Y^\ve_0=\cs^\ve,\;\; Y^\ve_i=v^\ve_i,\;\;1\leq i\leq m.
\]
Using (\ref{eq: cov0r}), (\ref{eq: cov1r}) and (\ref{eq: cov4r}) we deduce that for any $1\leq i, j\leq m$ we have
\[
\bsE(Y^\ve_0Y_i^\ve)=\cs_m\delta_{0i}+O(\ve),\;\;\bsE(Y^\ve_iY_j^\ve)= d_m\delta_{ij}+O(\ve^2).
\]
If $\bsS(Y^\ve)$ denotes the covariance operator of $Y$,  then we deduce that
\begin{equation}
\bsS(Y^\ve)^{-1}_{0,i}= \frac{1}{\cs_m}\delta_{0i}+O(\ve),\;\;\bsS(Y^\ve)^{-1}_{ij}=\frac{1}{d_m}\delta_{ij}+O(\ve). 
\label{eq: cov5}
\end{equation}
We now need to compute the covariance operator $\Cov(H^\ve,Y^\ve)$. To do so  we equip $\Sym_m$ with the inner product
\[
(A,B)=\tr(AB),\;\;A,B\in\Sym_m
\]
The space    $\Sym_m$ has a canonical orthonormal basis
\[
\widehat{E}_{ij},\;\;1\leq i\leq j\leq m,
\]
where 
\[
\widehat{\bsE}_{ij}= \begin{cases}
\bsE_{ij}, & i=j\\
\frac{1}{\sqrt{2}} \bsE_{ij}, & i<j
\end{cases}
\]
and   $\bsE_{ij}$ denotes the symmetric matrix     nonzero entries only at locations $(i,j)$ and $(j,i)$ and these entries are equal to $1$.  Thus a matrix $A\in\Sym_m$ can be written as
\[
A=\sum_{i\leq j} a_{ij} \bsE_{ij} =\sum_{i\leq j} \hat{a}_{ij} \widehat{\bsE}_{ij},
\]
where
\[
\widehat{a}_{ij}=  \begin{cases}
a_{ij}, &i=j,\\
\sqrt{2} a_{ij}, &i<j.
\end{cases}
\]
The covariance operator  $\Cov(H^\ve,Y\ve)$ is a linear map
\[
\Cov(H^\ve,Y^\ve):\bR\oplus \bR^m \ra \Sym_m
\]
given by
\[
\Cov(H^\ve,Y^\ve)\left(\sum_{\alpha=0}^m y_\alpha \be_\alpha \right) = \sum_{i<j, \alpha}  \bsE(\widehat{H}_{ij}^\ve  Y_\alpha^\ve)y_\alpha\widehat{\bsE}_{ij}=\sum_{i<j, \alpha}  \bsE({H}^\ve_{ij}  Y^\ve_\alpha)y_\alpha{\bsE}_{ij},
\]
where $\be_0,\be_1,\dotsc,\be_m$ denotes the canonical orthonormal basis in $\bR\oplus\bR^m$.  Using (\ref{eq: cov3r}) and (\ref{eq: cov4r}) we deduce  that
\begin{equation}
\Cov(H^\ve,Y^\ve)\left(\sum_{\alpha=0}^m y_\alpha \be_\alpha \right) = -y_0d_m \one_m+ O(\ve).
\label{eq: cov6}
\end{equation}
We deduce that  the transpose $\Cov(H^\ve,Y^\ve)\dual$ satisfies
\begin{equation}
\Cov(H^\ve,Y^\ve)\dual \left(\sum_{i\leq j} \hat{a}_{ij}\widehat{\bsE}_{ij}\right)=  -d_m\tr(A)\be_0+O(\ve).
\label{eq: cov6d}
\end{equation}
Set
\[
Z^\ve:= (H^\ve| \cs^\ve=x, v^\ve=0)-\bsE(H^\ve| \cs^\ve=x, v^\ve=0).
\]
Above,  $Z^\ve$ is a \emph{centered} Gaussian random matrix with covariance operator
\[
\bsS(Z^\ve)= \bsS(H^\ve) - \Cov(H^\ve,Y^\ve)\bsS(Y^\ve)^{-1} \Cov(H^\ve,Y^\ve)\dual.
\]
This means that
\[
\bsE\bigl(\, \widehat{z}_{ij}^\ve\widehat{z}_{k\ell}^\ve\,\bigr)= (\widehat{\bsE}_{ij}, \bsS(Z^\ve)\widehat{\bsE}_{k\ell}).
\]
Using (\ref{eq: cov5}), (\ref{eq: cov6}) and (\ref{eq: cov6d}) we deduce that
\[
\Cov(H^\ve,Y^\ve)\bsS(Y^\ve)^{-1} \Cov(H^\ve,Y^\ve)\dual\left(\sum_{i\leq j} \hat{a}_{ij}\widehat{\bsE}_{ij}\right)=\frac{d_m^2}{\cs_m}\tr(A)\one_m+O(\ve)
\]
\[
\bsE\bigl(\,(z_{ij}^\ve)^2\,\bigr)= h_m+ O(\ve), \;\;\bsE(z_{ii}^\ve z_{jj}^\ve)=h_m-\frac{d_m^2}{\cs_m}+ O(\ve),\;\;\forall i<j,
\]
\[
\bsE\bigr(\,(z_{ii}^\ve)^2\,\bigr)=3h_m-\frac{d_m^2}{\cs_m}+O(\ve),\;\; \forall i
\]
and 
\[
\bsE(z_{ij}^\ve z_{k\ell}^\ve)=O(\ve),\;\;\forall i<j,\;\;k\leq \ell,\;\;(i,j)\neq (k,\ell).
\]
We can rewrite these equalities in the compact form
\[
\bsE(z_{ij}^\ve z_{k\ell}^\ve )=\left(h_m-\frac{d_m^2}{s_m}\right)\delta_{ij}\delta_{k\ell} + h_m(\delta_{ik}\delta_{j\ell}+ \delta_{i\ell}\delta_{jk})+O(\ve).
\]
Note that
\[
h_m-\frac{d_m^2}{\cs_m}\stackrel{(\ref{eq: sdh1})}{=} \frac{r_m-1}{r_m}h_m.
\]
We set 
\[
\kappa_m:=\frac{(r_m-1)}{2r_m},
\]
 so that
\[
\bsE(z_{ij}^\ve z_{k\ell}^\ve)=2\kappa_m h_m\delta_{ij}\delta_{k\ell} + h_m(\delta_{ik}\delta_{j\ell}+ \delta_{i\ell}\delta_{jk})+O(\ve).
\]
Using (\ref{eq: cond-expect})  we deduce that 
\begin{equation}
\bsE(H^\ve| \cs^\ve=x, v^\ve=0)=\Cov(H^\ve,Y^\ve)\bsS(Y^\ve)^{-1}(x\be_0)= -\frac{xd_m}{\cs_m}\one_m+O(\ve).
\label{eq: expect}
\end{equation}
We deduce that the  Gaussian random matrix $(H^\ve| \cs^\ve=x, v^\ve=0)$  converges uniformly in $\bp$ as $\ve\to  0$ to the random matrix $A-\frac{x}{r_m(m+4)}\one_m$, where    $A$  belongs to the Gaussian ensemble $\Sym_m^{2\kappa h_m, h_m}$ described in   Appendix \ref{s: gmat}. Thus
\[
\lim_{\ve\to 0}q_{\ve,\bp}(B) =q_\infty(B):=\int_B \bsE_{\Sym_m^{2\kappa_m h_m, h_m}}\bigl(\, \bigl|\,\det\bigl(\,  A-\frac{xd_m}{\cs_m} \one_m\,\bigr)\,\bigr|\,\bigr)\frac{ e^{- \frac{x^2}{2\cs_m}}}{\sqrt{2\pi \cs_m}} dx
\]
\[
=(h_m)^{\frac{m}{2}} \int_B \bsE_{\Sym_m^{2\kappa_m, 1}}\bigl(\, \bigl|\,\det\bigl(\,  A-\frac{x}{\cs_m\sqrt{h_m}} \one_m\,\bigr)\,\bigr|\,\bigr)\frac{ e^{- \frac{x^2}{2\cs_m}}}{\sqrt{2\pi \cs_m}} dx
\]
\[
=(h_m)^{\frac{m}{2}} \int_{(\cs_m)^{-\frac{1}{2}}B} \bsE_{\Sym_m^{2\kappa_m, 1}}\bigl(\, \bigl|\,\det\bigl(\,  A-\alpha_m y \one_m\,\bigr)\,\bigr|\,\bigr)\frac{ e^{- \frac{y^2}{2}}}{\sqrt{2\pi }} dx,
\]
where
\[
\alpha_m=\frac{d_m}{\sqrt{\cs_mh_m}}\stackrel{(\ref{eq: sdh1})}{=}\frac{1}{\sqrt{r_m}}.
\]
This proves that
\[
\lim_{\ve\searrow 0} \eR_{(\cs_m)^{-\frac{1}{2}}}q_{\ve,\bp}(B)=(h_m)^{\frac{m}{2}} \underbrace{\int_{B} \bsE_{\eS_m^{2\kappa_m, 1}}\Bigl(\, \Bigl|\,\det\Bigl(\,  A-\frac{y}{\sqrt{r_m}} \one_m\,\Bigr)\,\Bigr|\,\Bigr)\frac{ e^{- \frac{y^2}{2}}}{\sqrt{2\pi }} dy}_{=:\mu_m(B)}.
\]
Using the last equality, the normalization assumption (\ref{tag: ast}) and the estimate  (\ref{eq: rhol}) in  (\ref{eq: dlb}) we conclude 
\begin{equation}
\bigl(\, \eR_{ (\cs_m\ve^{-m})^{ -\frac{1}{2}}}\,\bigr)_*\tsi^\ve(B)=\ve^{-m}\left(\,\left(\frac{h_m}{2\pi d_m}\right)^{\frac{m}{2}}\mu_m(B)+O(\ve)\,\right)\;\;\mbox{as $\ve \to 0$}.
\label{eq: limit}
\end{equation}
In particular
\begin{equation}
\bsN^\ve= \ve^{-m}\left(\,\left(\frac{h_m}{2\pi d_m}\right)^{\frac{m}{2}}\mu_m(\bR)+O(\ve)\,\right)\;\;\mbox{as $\ve \to 0$}.
\label{eq: bsN}
\end{equation}
Observe that the density of $\mu_m$ is
\begin{equation}
\frac{d\mu_m}{dy}= \bsE_{\Sym_m^{2\kappa_m, 1}}\Bigl(\, \Bigl|\,\det\Bigl(\,  A-\frac{y}{\sqrt{r_m}} \one_m\,\Bigr)\,\Bigr|\,\Bigr)\frac{ e^{- \frac{y^2}{2}}}{\sqrt{2\pi }}
\label{eq: dmu}
\end{equation}
($\tilde{A}=\sqrt{r_m} A$)
\[
=r_m^{-\frac{m}{2}}\bsE_{\Sym_m^{2\kappa_mr_m, r_m}}\Bigl(\, \Bigl|\,\det\Bigl(\,  \tilde{A}-y \one_m\,\Bigr)\,\Bigr|\,\Bigr)\frac{ e^{- \frac{y^2}{2}}}{\sqrt{2\pi }}
\]
($2k_mr_m=r_m-1$)
\[
\stackrel{(\ref{eq: expec_uv})}{=} r_m^{-\frac{m}{2}}2^{\frac{3}{2}}(2r_m)^{\frac{m+1}{2}}\Gamma\left(\frac{m+3}{2}\right)(\gamma_{r_m-1}\ast\theta^+_{m+1,r_m})(y)\gamma_1(y).
\]
\[
=2^{\frac{m+4}{2}}r_m^{\frac{1}{2}}\Gamma\left(\frac{m+3}{2}\right)(\gamma_{r_m-1}\ast\theta^+_{m+1,r_m})(y)\gamma_1(y).
\]
This proves part (a) and (\ref{eq: lim-si-a}) in Theorem \ref{th: main}. To prove (\ref{eq: lim-si-b})  we  distinguish two cases.

\medskip

\noindent {\bf Case 1.}  $r_m>1$.  From Lemma \ref{lemma: exp-det1} we deduce that
\begin{equation}
\begin{split}
   \bsE_{\Sym_m^{2\kappa_m, 1}}\Bigl(\, \Bigl|\,\det\Bigl(\,  A-\frac{y}{\sqrt{r_m}} \one_m\,\Bigr)\,\Bigr|\,\Bigr) \;\;\;\;\;\;\;\;\; \;\;\;\;\;\;\;\;\;\;\;& \\
  =2^{\frac{m+3}{2}}\Gamma\left(\frac{m+3}{2}\right)\frac{1}{\sqrt{2\pi\kappa_m}}\int_{\bR}\rho_{m+1,1}(\lambda)  e^{-\frac{1}{4\tau_m^2}(\lambda-(\tau_m^2+1)\frac{y}{\sqrt{r_m}})^2 +\frac{(\tau_m^2+1)y^2}{4r_m}} d\lambda  &  ,
\end{split}
\label{eq: exp-cond}
\end{equation}
where
\[
\tau_m^2:=\frac{\kappa_m}{\kappa_m-1}=\frac{r_m-1}{r_m+1}.
\]
Thus
\[
\frac{d\mu_m}{dy}= 2^{\frac{m+3}{2}}\Gamma\left(\frac{m+3}{2}\right)\frac{1}{\sqrt{2\pi\kappa_m}} e^{\frac{(\tau_m^2+1-2r_m)y^2}{4r_m}}  \int_{\bR}\rho_{m+1,1}(\lambda)  e^{-\frac{1}{4\tau_m^2}(\lambda-(\tau_m^2+1)\frac{y}{\sqrt{r_m}})^2 } d\lambda
\]
\[
= 2^{\frac{m+3}{2}}\Gamma\left(\frac{m+3}{2}\right)\frac{1}{\sqrt{2\pi\kappa_m}} \int_{\bR}\rho_{m+1,1}(\lambda)  e^{-\frac{1}{4\tau_m^2}(\lambda-(\tau_m^2+1)\frac{y}{\sqrt{r_m}})^2 -\frac{r_my^2}{2(r_m+1)}} d\lambda.
\]
An elementary computation  yields a pleasant surprise
\[
-\frac{1}{4\tau_m^2}\left(\lambda-(\tau_m^2+1)\frac{y}{\sqrt{r_m}}\right)^2 -\frac{r_my^2}{2(r_m+1)}= -\frac{1}{4}\lambda^2-\left( \sqrt{\frac{1}{2(r_m-1)}} \lambda -y \sqrt{\frac{r_m}{2(r_m-1)}}\right)^2.
\]
Now set
\[
\beta_m:= \frac{1}{(r_m-1)} .
\]
We deduce
\[
\frac{d\mu_m}{dy}=  2^{\frac{m+3}{2}}\Gamma\left(\frac{m+3}{2}\right)\frac{1}{2\pi\sqrt{\kappa_m}} \int_{\bR}\rho_{m+1,1}(\lambda) e^{-\frac{1}{4}\lambda^2} e^{-\frac{\beta_m}{2}(\lambda-\sqrt{r_m}y)^2} d\lambda.
\]
($\lambda:=\sqrt{r}\lambda$)
\[
= 2^{\frac{m+3}{2}}\Gamma\left(\frac{m+3}{2}\right)\frac{1}{\sqrt{2\pi\kappa_m}}  \int_{\bR}\sqrt{r_m}\rho_{m+1,1}(\sqrt{r_m}\lambda) e^{-\frac{r_m}{4}\lambda^2} e^{-\frac{r_m\beta_m}{2}(\lambda-y)^2} d\lambda
\]
\[
\stackrel{(\ref{eq: resc-cor})}{=}2^{\frac{m+3}{2}}\Gamma\left(\frac{m+3}{2}\right)\frac{1}{\sqrt{\kappa_m r_m\beta_m}}  \int_{\bR}\rho_{m+1,1/r_m}(\lambda) e^{-\frac{r_m}{4}\lambda^2} d\gamma_{\frac{1}{\beta_m r_m}}(y-\lambda) d\lambda.
\]
($\kappa_mr_m\beta_m=\frac{1}{2}$)
\[
= 2^{\frac{m+4}{2}}\Gamma\left(\frac{m+3}{2}\right)  \int_{\bR}\rho_{m+1,1/r_m}(\lambda) e^{-\frac{r_m}{4}\lambda^2} d\gamma_{\frac{1}{\beta_m r_m}}(y-\lambda) d\lambda
\]
\[
= 2^{\frac{m+4}{2}}\Gamma\left(\frac{m+3}{2}\right)\int_{\bR}\rho_{m+1,1/r_m}(\lambda) e^{-\frac{r_m}{4}\lambda^2} d\gamma_{\frac{r_m-1}{ r_m}}(y-\lambda) d\lambda
\]
\
Using the last equality in  (\ref{eq: limit})    we obtain the case $r_m>1$ (\ref{eq: lim-si-b}) of Theorem \ref{th: main}.

\medskip

\noindent {\bf Case 2.} $r_m=1$.  The  proof of  Theorem \ref{th: main} in  this case follows a similar pattern.  Note first that in this case $\kappa_m=0$ so  invoking Lemma \ref{lemma: exp-det} we obtain the  following counterpart of (\ref{eq: exp-cond}) 
\[
 \bsE_{\GOE_m^1}\Bigl(\, \Bigl|\,\det\Bigl(\,  A-y \one_m\,\Bigr)\,\Bigr|\,\Bigr)=2^{\frac{m+4}{2}}\Gamma\left(\frac{m+3}{2}\right) e^{\frac{y^2}{4}}\rho_{m+1,1}(y).
 \]
Using this in (\ref{eq: dmu}) we deduce 
\[
\frac{d\mu_m}{dy}= 2^{\frac{m+4}{2}}\Gamma\left(\frac{m+3}{2}\right) e^{\frac{-y^2}{4}}\rho_{m+1,1}(y)
\]
which is  (\ref{eq:  lim-si-b}) in the case $r_m=1$.  This completes the proof of Theorem \ref{th: main}. \qed

\subsection{Proof of Corollary \ref{cor: main}}  According to (\ref{eq: tsi}) we have
\[
\gamma_{\bom_m\ve^{-m}}\ast \bsi^\ve =\tsi^\ve.
\]
Thus
\[
\gamma_{\frac{\bom_m}{\cs_m}}\ast \Bigl(\eR_{\frac{1}{\sqrt{\cs_m\ve^{-m}} }}\Bigr)_*\bsi^\ve=\Bigl(\eR_{\frac{1}{ \sqrt{\cs_m\ve^{-m}} }}\Bigr)_*\tsi^\ve.
\]
Hence
\[
\lim_{\ve\to 0}\frac{1}{\bsN^\ve}\gamma_{\frac{\bom_m}{\cs_m}}\ast \Bigl(\eR_{\frac{1}{\sqrt{\cs_m\ve^{-m}} }}\Bigr)_*\bsi^\ve= \tsi_m.
\]
We can now conclude by invoking L\'{e}vy's continuity theorem \cite[Thm. 15.23(ii)]{Kl} or \cite[Thm. 2.4]{Var}.\qed

\subsection{Proof of Theorem \ref{th: main1}}

We have 
\begin{equation}
\tsi_m= \frac{1}{K_m} \theta^-_{m+1,\frac{1}{r_m}}\ast \gamma_{\frac{r_m-1}{r_m}} dy
\label{eq: tsi-asy}
\end{equation}
where
\[
\theta^-_{m+1,\frac{1}{r_m}}(\lambda)=\rho_{m+1,\frac{1}{r_m}}(\lambda)e^{-\frac{r_m\lambda^2}{4}},
\]
and
\[
K_m= \int_{\bR} \theta^-_{m+1,\frac{1}{r_m}}\ast \gamma_{\frac{r_m-1}{r_m}}(y) dy=  \int_{\bR} \theta^-_{m+1,\frac{1}{r_m}}(\lambda) d\lambda =\int_\bR \rho_{m+1,\frac{1}{r_m}}(\lambda)e^{- \frac{r_m\lambda^2}{4}} d\lambda.
\]
We set
\[
R_m(\lambda):=\rho_{m+1,\frac{1}{m}}(\lambda),\;\;R_\infty(x)=\frac{1}{2\pi} \bsI_{\{|x|\leq 2\}}\sqrt{4-x^2}. 
\]
Fix $c\in (0,2)$.  In \cite[\S 4.2]{N2} we proved that
\begin{subequations}
\begin{equation}
\lim_{m\ra \infty}\sup_{|x|\leq c} |\bar{R}_m(x)-R_\infty(x)|=0,
\label{eq: est2}
\end{equation}
\begin{equation}
\sup_{|x|\geq c} |\bar{R}_m(x)-R_\infty(x)|= O(1)\;\;\mbox{as $m\to\infty$}.
\label{eq: est3}
\end{equation}
\end{subequations}
Then
\[
\rho_{m+1,\frac{1}{r_m}}(\lambda)=\sqrt{\frac{r_m}{m}} R_m\left(\sqrt{\frac{r_m}{m}}\lambda\right),\;\;\theta^-_{m+1,\frac{1}{r_m}}(\lambda)= \sqrt{\frac{r_m}{m}} R_m\left(\sqrt{\frac{r_m}{m}}\lambda\right)e^{-\frac{r_m\lambda^2}{4}}.
\]
We now distinguish two cases.

\medskip

\noindent {\bf Case 1.} $r=\lim_{m\to\infty }r_m <\infty$. In particular, $r\in [1,\infty)$.  In this case we have
\[
K_m  =\sqrt{\frac{r_m}{m}} \int_\bR R_m\left(\sqrt{\frac{r_m}{m}}\lambda\right)e^{-\frac{r_m\lambda^2}{4}} d\lambda,
\]
 and using (\ref{eq: est2})-(\ref{eq: est3}) we deduce
 \[
 \lim_{m\to \infty} \int_\bR R_m\left(\sqrt{\frac{r_m}{m}}\lambda\right)e^{-\frac{r_m\lambda^2}{4}} d\lambda=  R_\infty(0)\int_\bR e^{-\frac{r\lambda^2}{4}} dr =R_\infty(0)\sqrt{\frac{4\pi}{r}}.
 \]
 Hence
 \begin{equation}
 K_m\sim K'_m=R_\infty(0)\sqrt{\frac{4\pi}{ m}}\;\;\mbox{as $m\to \infty$}.
 \label{eq: normali}
 \end{equation}
Now observe that
\[
\frac{1}{K'_m}\theta^-_{m+1,\frac{1}{r_m}}(\lambda)d\lambda=\frac{1}{R_\infty(0)}  R_m\left(\sqrt{\frac{r_m}{m}}\lambda\right) \frac{r_m}{\sqrt{4\pi}} e^{-\frac{r_m\lambda^2}{4} }d\lambda
\]
 \[
 = \frac{1}{R_\infty(0)}  R_m\left(\sqrt{\frac{r_m}{m}}\lambda\right)  \gamma_{\frac{2}{r_m}}(d\lambda)
 \]
 Using (\ref{eq: est2}) and (\ref{eq: est3}) we conclude that the  sequence of measures
  \[
 \frac{1}{K'_m}\theta^-_{m+1,\frac{1}{r_m}}(\lambda)d\lambda
 \]
 converges weakly to  the Gaussian measure $\gamma_{\frac{2}{r}}$.     Using this and the    asymptotic equality (\ref{eq: normali})  in (\ref{eq: tsi-asy}) we deduce
 \[
 \lim_{m\to\infty}\tsi_m=\gamma_{\frac{2}{r}}\ast \gamma_{\frac{r-1}{r}}=\gamma_{\frac{r+1}{r}}.
 \]
 This proves Theorem \ref{th: main1} in the case $r<\infty$ since
 \[
 \gamma_{\frac{\bom}{\cs_m}}\ast \bsi_m=\tsi_m\;\;\mbox{and}\;\;\lim_{m\to\infty} \frac{\bom_m}{\cs_m}\stackrel{(\ref{eq: rq})}{=}0.
 \]
\noindent{\bf Case 2.} $\lim_{m\to\infty} r_m=\infty$.  In this case we write
\[
\theta^-_{m+1,\frac{1}{r_m}}(\lambda)d\lambda =\sqrt{\frac{4\pi}{m}}  R_m\left(\sqrt{\frac{r_m}{m}}\lambda\right)\gamma_{\frac{2}{r_m}}(\lambda)d\lambda.
\]
\begin{lemma}  The sequence of measures
\[
 R_m\left(\sqrt{\frac{r_m}{m}}\lambda\right)\gamma_{\frac{2}{r_m}}(\lambda)d\lambda
 \]
 converges weakly to  the measure $R_\infty(0)\delta_0$.
 \label{lemma: w-conv}
 \end{lemma}
 
 \begin{proof}  Fix a  bounded continuous function $f:\bR\to \bR$.  Observe first that
 \begin{equation}
 \lim_{m\to\infty} \underbrace{ \int_\bR \left(R_m\left(\sqrt{\frac{r_m}{m}}\lambda\right)-R_\infty\left(\sqrt{\frac{r_m}{m}}\lambda\right)\right)f(\lambda)\gamma_{\frac{2}{r_m}}(\lambda)d\lambda}_{=D_m}=0.
 \label{eq: lim-diff}
 \end{equation}
 Indeed, we have
 \[
 D_m= \underbrace{\int_{|\lambda|<c\frac{\sqrt{m}}{\sqrt{r_m}}} \left(R_m\left(\sqrt{\frac{r_m}{m}}\lambda\right)-R_\infty\left(\sqrt{\frac{r_m}{m}}\lambda\right)\right)f(\lambda)\gamma_{\frac{2}{r_m}}(\lambda)d\lambda}_{=:D_m'}
 \]
 \[
 +\underbrace{\int_{|\lambda|>c\frac{\sqrt{m}}{\sqrt{r_m}}} \left(R_m\left(\sqrt{\frac{r_m}{m}}\lambda\right)-R_\infty\left(\sqrt{\frac{r_m}{m}}\lambda\right)\right)f(\lambda)\gamma_{\frac{2}{r_m}}(\lambda)d\lambda}_{=:D_m''}.
 \]
 Observe that
 \[
 D_m'\leq \sup_{|x|\leq c} |R_m(x)-R_\infty(x)| \int_{|\lambda|<c\frac{\sqrt{m}}{\sqrt{r_m}}}f(\lambda)\gamma_{\frac{2}{r_m}}(\lambda)d\lambda
 \]
 and invoking (\ref{eq: est2}) we deduce 
 \[
 \lim_{m\to\infty} D_m'=0.
 \]
 Using (\ref{eq: est3}) we deduce that there exists a constant  $S>0$ such that
 \[
 D_m'\leq S \int_{|\lambda|>c\frac{\sqrt{m}}{\sqrt{r_m}}}  \gamma_{\frac{2}{r_m}}(\lambda)d\lambda.
  \]
 On the other hand, Chebyshev's inequality shows that
 \[
 \int_{|\lambda|>c\frac{\sqrt{m}}{\sqrt{r_m}}}  \gamma_{\frac{2}{r_m}}(\lambda)d\lambda\leq \frac{2}{c^2m}.
 \]
 Hence
 \[
 \lim_{m\to\infty} D_m''=0.
 \]
 This proves  (\ref{eq: lim-diff}).      
 
 The sequence of measures $\gamma_{\frac{2}{r_m}}(\lambda)d\lambda$ converges to $\delta_0$ so that
 \[
 R_\infty(0) f(0) = \lim_{m\to\infty} \int_\bR R_\infty(0)f(\lambda) \gamma_{\frac{2}{r_m}}(\lambda)d\lambda.
 \]
 Using (\ref{eq: lim-diff}) and the above  equality we deduce that the conclusion of the lemma is equivalent to
 \begin{equation}
 \lim_{m\to\infty} \underbrace{ \int_\bR \left(R_\infty(0)-R_\infty\left(\sqrt{\frac{r_m}{m}}\lambda\right)\right)f(\lambda)\gamma_{\frac{2}{r_m}}(\lambda)d\lambda}_{=F_m}=0.
 \label{eq: lim-diff1}
 \end{equation}
 To prove this  we decompose $F_m$ as follows.
 \[
 F_m= \underbrace{\int_{|\lambda|<m^{-\frac{1}{4}}\frac{\sqrt{m}}{\sqrt{r_m}}} \left(R_\infty(0)-R_\infty\left(\sqrt{\frac{r_m}{m}}\lambda\right)\right)f(\lambda)\gamma_{\frac{2}{r_m}}(\lambda)d\lambda}_{=:F_m'}
 \]
 \[
 +\underbrace{\int_{|\lambda|>m^{-\frac{1}{4}}\frac{\sqrt{m}}{\sqrt{r_m}}} \left(R_\infty(0)-R_\infty\left(\sqrt{\frac{r_m}{m}}\lambda\right)\right)f(\lambda)\gamma_{\frac{2}{r_m}}(\lambda)d\lambda}_{=:F_m''}.
 \]
Observe that
 \[
 F_m'\leq \sup_{|x|\leq m^{-\frac{1}{4}} } |R_\infty(0)-R_\infty(x)| \int_{|\lambda|<m^{-\frac{1}{4}} \frac{\sqrt{m}}{\sqrt{r_m}}}f(\lambda)\gamma_{\frac{2}{r_m}}(\lambda)d\lambda
 \]
 and since $R_\infty$ is continuous at $0$ we deduce 
 \[
 \lim_{m\to\infty} F_m'=0.
 \]
Since $R_\infty$ and $f$ are bounded we deduce that there exists a constant $S>0$ such that
\[
 F_m''\leq S \int_{|\lambda|>m^{-\frac{1}{4}} \frac{\sqrt{m}}{\sqrt{r_m}}}  \gamma_{\frac{2}{r_m}}(\lambda)d\lambda.
  \]
 On the other hand, Chebyshev's inequality shows that
 \[
 \int_{|\lambda|>m^{-\frac{1}{4}} \frac{\sqrt{m}}{\sqrt{r_m}}}  \gamma_{\frac{2}{r_m}}(\lambda)d\lambda\leq \frac{2}{\sqrt{m}}.
 \]
 Hence
 \[
 \lim_{m\to\infty} F_m''=0.
 \]
 This proves  (\ref{eq: lim-diff1}) and the lemma.     
 \end{proof}

 Lemma \ref{lemma: w-conv} shows that
 \[
 K_m\sim K'_m  =\sqrt{\frac{4\pi}{m}} R_\infty(0),
 \]
 and
 \[
 \lim_{m\to\infty}\frac{1}{K_m}\theta^-_{m+1,\frac{1}{r_m}}(\lambda)d\lambda=\delta_0.
 \]
 On the other hand
 \[
 \lim_{m\to\infty} \gamma_{\frac{r_m-1}{r_m}}(\lambda) d\lambda =\gamma_1(\lambda) d\lambda,
 \]
 so that 
 \[
 \lim_{m\to\infty} \tsi_m=\delta_0\ast\gamma_1=\gamma_1.
 \]
 This completes the proof of Theorem \ref{th: main1}.  \qed

\subsection{Proof of  Corollary \ref{cor: main1}.} Using (\ref{eq: bsN})  we deduce
 \[
 \ve^m\bsN^\ve = \left(\frac{h_m}{2\pi d_m}\right)^{\frac{m}{2}}\mu_m(\bR)+O(\ve)
 \]
 \[
 = 2^{\frac{m+4}{2}}\Gamma\left(\frac{m+3}{2}\right)\left(\frac{h_m}{2\pi d_m}\right)^{\frac{m}{2}}\int_\bR\theta^-_{m+1,\frac{1}{r_m}}\ast \gamma_{\frac{r_m-1}{r_m}}(y) dy +O(\ve)
 \]
 \[
 =2^{\frac{m+4}{2}}\Gamma\left(\frac{m+3}{2}\right)\left(\frac{h_m}{2\pi d_m}\right)^{\frac{m}{2}}\int_\bR \theta^-_{m+1,\frac{1}{r_m}}(\lambda) d\lambda + O(\ve)
 \]
 \[
 =\underbrace{2^{\frac{m+4}{2}}\Gamma\left(\frac{m+3}{2}\right)\left(\frac{h_m}{2\pi d_m}\right)^{\frac{m}{2}}K_m}_{=C_m(w)} +O(\ve).
 \]
 Lemma \ref{lemma: w-conv} implies that as $m\to \infty$ we have
 \[
 K_m\sim \sqrt{\frac{4\pi}{m}}R_\infty(0)=\frac{2}{\sqrt{\pi m}}.
 \]
 We deduce that 
 \[
 C_m(w)\sim  \frac{2^{\frac{m+6}{2}}}{\sqrt{\pi m}}\Gamma\left(\frac{m+3}{2}\right)\left(\frac{h_m}{2\pi d_m}\right)^{\frac{m}{2}}\;\;\mbox{as $m\to \infty$}.
 \]
 \qed

 \subsection{Proof of  Theorem \ref{th: metric}.}   Fix  a  point $\bp\in M$ and normal coordinates $(x^i)$ near $\bp$. The  equality (\ref{eq: cov1}) shows that   as $\ve\to 0$ we have the following estimate, uniform in $\bp$.
 \[
 \bsE\bigl(\, \pa_{x^i}\tu_\ve(\bp)\pa_{x^j}\tu_\ve(\bp)\,\bigr)= d_m\ve^{-(m+2)}\bigl(\delta_{ij}+O(\ve)^2)\,\bigr).
 \]
 Hence
 \begin{equation}
 h^\ve(\pa_{x^i},\pa_{x^j})= \delta_{ij}+O(\ve^2)= g_\bp(\pa_{x^i},\pa_{x^j})+O(\ve^2).
 \label{eq: hve}
\end{equation}
 This proves     (a) and (b) of Theorem \ref{th: main}.

 With $\bp$ and $(x^i)$ as above  we  set
\[
\begin{split}
\eE^\ve_{i_1,\dotsc, i_a; j_1,\dotsc, j_b}:=\frac{\pa^{a+b}\eE^\ve(x,y)}{\pa x^{i_1}\cdots \pa x^{i_a}\pa y^{j_1}\cdots \pa y^{j_b}}|_{x=y=0},\\
\\
h^\ve_{ij}:= h^\ve_\bp(\pa_{x^i},\pa_{x^j}),\;\;1\leq i,j\leq m.
\end{split}
\]
We denote by $K^\ve_{ij}$ the sectional curvature  of $h^\ve$ along the plane spanned by  $\pa_{x^i},\pa_{x^j}$.  Using \cite[Lemma 12.2.1]{AT} and that the sectional curvatures of a metric  are inverse proportional to the metric we deduce as in \cite[\S 3.3]{N2} that
\[
K_{ij}^\ve=\frac{d_m}{\ve^{m+2}}\times \frac{\eE^\ve_{ii;jj}-\eE^\ve_{ij;ij}}{\eE^\ve_{i;i}\eE^\ve_{j,j}-(\eE^\ve_{i;j})^2}.
\]
Using  Theorem \ref{prop: two-step}  we deduce that there exists a  universal constant $\eZ_m$ that depends only on $m$ and $w$ such that
\begin{equation}
\eE^\ve_{ii;jj}-\eE^\ve_{ij;ij}= \ve^{-(m+2)}\eZ_m K_{ij}(\bp)\bigl(\,1+ O(\ve^2)\,\bigr),
\label{kijve}
\end{equation}
where $K_{ij}(\bp)$ denotes the sectional curvature of $g$ at $\bp$. The estimate (\ref{eq: cov1}) implies that
\[
\eE^\ve_{i;i}\eE^\ve_{j,j}-(\eE^\ve_{i;j})^2=d_m^2\ve^{-2(m+2)} \bigl(\, 1+O(\ve^2)\,\bigr).
\]
Thus 
\[
K_{ij}^\ve=\frac{\eZ_m}{d_m}K_{ij}(\bp)\bigl(\,1+ O(\ve^2)\,\bigr).
\]
To determine  the  constant $\frac{\eZ_m}{d_m}$  it suffices to compute it  on a special manifold.    Assume that $M$ is the unit sphere $S^m$ equipped with  the round metric.  This is  is a homogeneous space equipped with an invariant metric $g$ with positive sectional curvatures.  The metrics $h^\ve$ are also invariant  so there exists a constant  $C_\ve>0$ such that $h^\ve=C_\ve g$. The estimate  (\ref{eq: hve}) implies that $C_\ve=1$ and thus $K_{ij}^\ve=K_{ij}(\bp)$ so that $\frac{\eZ_m}{d_m}=1$. \qed

 \section{Some examples}
 \label{s: 3}
 \setcounter{equation}{0}

We want to discuss several examples of weights $w$ satisfying the assumptions of the central limit theorem, Theorem \ref{th: main1}.   Observe first that
\[
r_n(w)\sim R_m(w)=\frac{I_{m-1}(w) I_{m+3}(w)}{I_{m+1}(w)}\;\;\mbox{as $m\to\infty$}.
\]
Moreover
\[
R_n(w_\ve)=R_n(w).
\]

\begin{ex} Suppose that  $w(t)=e^{-t^2}$. In this case   $\eE^\ve$ is the Schwartz kernel of the heat operator $e^{-\ve\Delta}$ whose  asymptotics as $\ve\to 0$ have been thoroughly investigated.  The momenta (\ref{eq: Ik}) are
\[
I_k(w)=\int_0^\infty t^k e^{-t^2} dt= \frac{1}{2}\int_0^\infty s^{\frac{k-1}{2}} e^{-s} ds=\frac{1}{2}\Gamma\left(\frac{k+1}{2}\right).
\]
Hence
\[
R_m(w) = \frac{\Gamma(\frac{m}{2})\Gamma(\frac{m}{2}+2)}{\Gamma(\frac{m}{2}+1)^2} = \frac{m+4}{m+2} \geq 1,\;\;q_m=\frac{m(m+4)}{(m+2)^2} < 1,\;\;\forall m
\]
so that $r_m=1$ for all $m$.  Moreover, in this case we have
\[
\frac{I_{m+3}(w)}{I_{m+1}(w)}=m+2,
\]
so that
\[
C_m(w)\sim \frac{2^{\frac{m+6}{2}}}{\sqrt{m}\pi^{\frac{m+1}{2}}}\Gamma\left(\frac{m+3}{2}\right)\;\;\mbox{as $m\to \infty$},
\]
and Stirling's formula implies
\begin{equation}
\log C_m(w)\sim \frac{m}{2}\log m\;\; \mbox{as $m\to \infty$}.
\label{eq: logcm}
\end{equation}\qed
\label{ex: gaussw}
\end{ex}

\begin{ex}  Suppose that
\[
w(t)= \exp\bigl( -(\log t)\log(\log t)\,\bigr),\;\;\forall t\geq 1.
\]
Observe that
\[
I_k(w)=\int_0^1 r^k w(r) dr +\int_1^\infty r^k\exp\bigl( -(\log r)\log(\log r)\,\bigr) dr.
\]
This proves that
\[
I_k(w)\sim  J_k:=\int_1^\infty r^k\exp\bigl( -(\log r)\log(\log r)\,\bigr) dr\;\;\mbox{as $k\to\infty$}.
\]
Using the substitution  $r=e^t$ we deduce
\[
J_k =\int_0^\infty e^{(k+1)t-t\log t} dt.
\]
We  want to investigate the large $\lambda$ asymptotics of the   integral
\begin{equation}
T_\lambda=\int_0^\infty e^{-\phi_\lambda(t)} dt,\;\;\phi_\lambda(t)=\lambda t- t\log t.  
\label{eq: T}
\end{equation}
We will achieve this by relying on the   Laplace method \cite[Chap. 4]{Bruj}. Note that
\[
\phi_\lambda'(t)=\lambda -\log t -1,\;\; \phi_\lambda''(t)=-\frac{1}{t}. 
\]
Thus $\phi_\lambda(t)$ has a unique critical point
\[
\tau=\tau(\lambda):= e^{\lambda-1}. 
\]
We make the change in variables $t=\tau s$ in (\ref{eq: T}). Observe that
\[
\lambda e^{\lambda-1}s-e^{\lambda-1}s \log(e^{\lambda-1} s)= e^{\lambda-1}s-(\lambda-1)e^{\lambda-1}s -e^{\lambda-1} \log s =e^{\lambda-1}s(1-\log s)
\]
and we deduce
\[
T_\lambda= \tau \int_0^\infty e^{-\tau h(s)} ds,\;\;h(s) = s(\log s-1).
\]
The  asymptotics of the last integral can be determined   using  the Laplace method and we have, \cite[\S 4.1]{Bruj}
\[
T_\lambda \sim \tau e^{-\tau h(1)} \sqrt{\frac{2\pi}{\tau h''(1)}}=\sqrt{2\pi\tau} e^\tau. 
\]
Hence
\[
J_k= T_{k+1}  \sim \sqrt{2\pi\tau(k+1)} e^{\tau(k+1)}=\sqrt{2\pi e^{k}}  e^{e^{k}}\;\;\mbox{as $k\to\infty$}.
\]
In this case
\[
R_m(w)\to \infty\;\;\mbox{as $m\to \infty$}.
\]
Note that
\[
\frac{h_m}{d_m}=\frac{2I_{m+3}(w)}{(m+2) I_{m+1}(w)}.
\]
We deduce that
\[
\log\left(\frac{h_m}{d_m}\right)\sim  e^{m+4} -e^{m+2}\;\;\mbox{as $m\to \infty$}.
\]
Hence
\begin{equation}
\log C_m(w) \sim \frac{m}{2}e^{m+2}(e^2-1)\;\;\mbox{as $m\to \infty$}.
\label{cminfi}
\end{equation}
\qed
\label{ex: h-tail}
\end{ex}

\begin{ex} Suppose that
\[
w(r)= \exp\bigl(-C(\log r)^\alpha\;\bigr), \;\; C>0, \;\;r>1,\;\;\alpha>1.
\]
Arguing as in Example \ref{ex: h-tail} we deduce that  as $k\to \infty$
\[
I_k(w)\sim\int_1^\infty r^k  \exp\bigl(\;-C(\log r)^\alpha\;\bigr)dr=\int_0^\infty e^{(k+1) t-C t^\alpha} dt. 
\]
Again, set
\[
T_\lambda := \int_0^\infty e^{-\phi_\lambda(t)} dt,\;\;\phi_\lambda(t):=Ct^\alpha-\lambda t. 
\]
We determine the  asymptotics of $T_\lambda$ as $\lambda\to\infty$ using the Laplace method. Note that
\[
\phi_\lambda'(t)= \alpha Ct^{\alpha-1} -\lambda. 
\]
The function $\phi_\lambda$ has a unique critical point
\[
\tau=\tau(\lambda)=\left(\frac{\lambda}{\alpha C}\right)^{\frac{1}{\alpha-1}}.
\]
Observe that
\[
\phi_\lambda(\tau s)=a(s^\alpha-bs),\;\; a:=\left(\frac{\lambda}{C^{1/\alpha}\alpha}\right)^{\frac{\alpha}{\alpha-1}},\;\;b:= \alpha^{\frac{1}{\alpha-1}},
\]
\[
T_\lambda= \tau(\lambda)\int_0^\infty  e^{-a(s^\alpha-bs)} ds.
\]
We set $g(s): =s^\alpha-bs.$  Using the Laplace method \cite[\S 4.2]{Bruj} we deduce
\[
T_\lambda\sim \tau(\lambda)e^{-a g(1)}\sqrt{ \frac{2\pi}{a g''(1)}} = \sqrt{ \frac{2\pi}{a \alpha(\alpha-1)}} e^{a(b-1)}.
\]
Hence
\[
\log T_\lambda \sim \left(\frac{\lambda^\alpha}{C}\right)^{\frac{1}{\alpha-1}}\frac{\alpha^{\frac{1}{\alpha-1}}-1}{\alpha^{\frac{\alpha}{\alpha-1}}}=:  Z(\alpha, C)\lambda^{\frac{\alpha}{\alpha-1}}.
\]
Hence
\[
\log R_m(w) \sim \log T_m+\log T_{m+4}-2\log T_{m+2}
\]
\[
\sim Z(\alpha, C)\Bigl(\,m^{\frac{\alpha}{\alpha-1}}+ (m+4)^{\frac{\alpha}{\alpha-1}}-2(m+2)^{\frac{\alpha}{\alpha-1}}\,\Bigr)
\]
\[
=Z(\alpha, C)m^{\frac{\alpha}{\alpha-1}}\left(\, 1+ \Bigl(1+\frac{4}{m}\Bigr)^{\frac{\alpha}{\alpha-1}}-2\Bigl(1+\frac{2}{m}\,\Bigr)^{\frac{\alpha}{\alpha-1}}\,\right)
\]
\[
\sim  Z(\alpha, C)m^{\frac{\alpha}{\alpha-1}}\times \frac{8}{m^2}\times \frac{\alpha}{\alpha-1}\Bigl(\frac{\alpha}{\alpha-1}-1\Bigr)=\frac{8\alpha Z(\alpha)}{(\alpha-1)^2} m^{\frac{2-\alpha}{\alpha-1}}.
\]
Hence
\begin{equation}
r=\lim_{m\to \infty}r_m=  \times\begin{cases}
\infty, & \alpha<2,\\
e^{16 Z(2, C)}, &\alpha=2,\\
1, & \alpha>2.
\end{cases}
\label{r2}
\end{equation}
which shows that $r$ can have any value in $[1,\infty]$. Note that in this case
\[
\log I_{m+3}(w)-\log I_{m+1}(w)\sim Z(\alpha, C) m^{\frac{\alpha}{\alpha-1}}\left(\,\Bigl(1+\frac{4}{m}\Bigr)^{\frac{\alpha}{\alpha-1}}-\Bigl(1+\frac{2}{m}\,\Bigr)^{\frac{\alpha}{\alpha-1}}\,\right)
\]
\[
\sim \frac{2Z(\alpha, C)}{\alpha-1}m^{\frac{1}{\alpha-1}},\;\;m\to\infty,
\]
so that
\begin{equation}
\log C_m(w)\sim \frac{Z(\alpha, C)}{\alpha-1}m^{\frac{\alpha}{\alpha-1}},\;\;m\to\infty.
\label{cm2}
\end{equation}
\qed
\end{ex}

\begin{ex}Suppose now that  $w$ is a weight  with compact support disjoint from the origin.   For example, assume that on the positive semi-axis it is given by
\[
w(x)=\begin{cases}
e^{-\frac{1}{1-(x-c)^2}}, & |x-c|\leq 1\\
0, & |x-c|>1,
\end{cases},\;\; c>1.
\]
Then
\[
I_k(w)=\int_{c-1}^{c+1} t^k e^{-\frac{1}{1-(t-c)^2}} dt= \int_{-1}^1 (t+c)^k e^{-\frac{1}{1-t^2}} dt 
\]
\[
=\underbrace{\int_{-1}^0(t+c)^k e^{-\frac{1}{1-t^2}} dt }_{I_k^-}  +\underbrace{\int_0^1 (t+c)^k e^{-\frac{1}{1-t^2}} dt }_{I_k^+}.
\]
Observe that
\[
\lim_{k\to\infty} c^{-k}I_k^-=0.
\]
On the other hand
\[
I_k^+=\int_0^1 (c+1-t)^k e^{-\frac{1}{t^2}} dt,
\]
and we   deduce
\[
 c^k\int_0^1 e^{-\frac{1}{t^2}} dt\leq I_k^+\leq (c+1)^k\int_0^1e^{-\frac{1}{t^2}} dt.
 \]
Hence the asymptotic behavior of $I_k(w)$ is determined by $I_k^+$. We will determine the asymptotic behavior of $I_k^+$ by relying again on the Laplace method. Set $a:= (c+1)$ so that
\[
I_k^+=\int_0^1(a-t)^ke^{-\frac{1}{t^2}} dt= a^k\int_0^{\frac{1}{a}} (1-s)^k e^{-\frac{1}{a^2s^2}} ds=a^k\int_a^\infty(u-1)^k u^{-(k+2)} e^{-\frac{u^2}{a^2}} du.
\]
Consider the phase 
 \[
 \phi_\hbar(s)= \frac{1}{\hbar}\log (1-s) -\frac{1}{a^2s^2},\;\;\hbar\searrow 0,
 \]
and set
\[
P_\hbar= a^{\frac{1}{\hbar}}\int_0^{\frac{1}{a}} e^{\phi_\hbar(s)} 
\]
so that
\[
I_k^+=P_{1/k}.
\]
 We have
 \[
 \phi_\hbar'(s)=-\frac{1}{\hbar(1-s)}+\frac{2}{a^2s^3},\;\;\phi_\hbar''(t)=-\frac{1}{\hbar(1-s)^2}-\frac{6}{a^2s^4}.
 \]
The phase $\phi_\hbar$ as a unique critical point  $\tau=\tau(\hbar)\in (0,1/a)$ satisfying
\[
\hbar= \frac{a^2\tau^3}{2(1-\tau)}=\frac{a^2\tau^3}{2}\bigl(1 +O(\tau)),
\]
so that
\begin{equation}
\tau =\left(\frac{2\hbar}{a^2}\right)^{\frac{1}{3}}\Bigl( 1+O\bigl(\,\hbar^{\frac{1}{3}}\,\bigr)\,\Bigr)\;\;\mbox{as $\hbar\searrow 0$}.
\label{eq: tauh}
\end{equation}
Set
\begin{equation}
v:=v(\hbar):=-\frac{1}{\phi_k''(\tau)} \sim \frac{a^2\tau^4}{6} \sim \frac{(2\hbar)^{\frac{4}{3}}}{6a^{\frac{2}{3}}}=\frac{1}{6}\left(\frac{2\hbar^2}{a}\right)^{\frac{2}{3}}.
\label{eq: vh}
\end{equation}
 We make the change in variables $s=\tau+\sqrt{v} x$ and we deduce
 \[
 P_\hbar= e^{\phi_\hbar(\tau)}a^{\frac{1}{\hbar}}\sqrt{v}\int_{J(\hbar)} e^{\phi_\hbar(\tau+\sqrt{v}x)-\phi_\hbar(\tau)} dx,\;\;J(\hbar)=\Bigl[\,-\frac{\tau}{\sqrt{v}}, \frac{1/a-\tau}{\sqrt{v}}\,\Bigr].
 \]
 We claim that
 \begin{equation}
 \lim_{\hbar\to 0} \int_{J(\hbar)} e^{\phi_\hbar(\tau+\sqrt{v}x)-\phi_\hbar(\tau)} dx=\int_{\bR}e^{-\frac{x^2}{2}} dx=\sqrt{2\pi}.
 \label{eq: asymp*}
 \end{equation}
 It is convenient to think of $\tau$ as the small parameter and then redefine 
\[
\hbar=\hbar(\tau)=\frac{a^2\tau^3}{2(1-\tau)}
\]
and think of $v$ as a function of $\tau$.  Finally set $\si:=\sqrt{v}$ and
\[
\vfi_\tau(x):=\phi_{\hbar(\tau)}(\tau+\si x)-\phi_{\hbar(\tau)}(\tau)=\frac{2(1-\tau)}{a^2\tau^3}\log(1-s)-\frac{1}{a^2s^2}
\]
\[
= \frac{2(1-\tau)}{a^2\tau^3}\Bigl(\,\log(1-\tau-\si x)-\log(1-\tau)\,\Bigr)-\frac{1}{a^2}\left(\frac{1}{(\tau+\si x)^2}-\frac{1}{\tau^2}\right)
\]
\[
=\frac{2(1-\tau)}{a^2\tau^3}\log\left(1-\frac{\si}{1-\tau} x\,\right)-\frac{1}{a^2\tau^2}\left(\frac{1}{(1+\frac{\si}{\tau}x)^2}-1\right)
\]
\[
=\frac{1}{a^2\tau^2}\left(\frac{2(1-\tau)}{\tau}\log\left(1-\frac{\si}{1-\tau} x\,\right)-\left(\frac{1}{(1+\frac{\si}{\tau}x)^2}-1\right)\right).
\]
The equality (\ref{eq: asymp*}) is equivalent to
\begin{equation}
\lim_{\tau\to \infty}\int_{J(\hbar)} e^{\vfi_\tau(x)} =\int_{\bR}e^{-\frac{x^2}{2}} dx.
\label{eq: asymptau}
\end{equation}
By construction, we have
\[
\vfi_\tau(0)=\vfi'_\tau(0)=0,\;\;\vfi''_\tau(0)=-1,\;\;\vfi_\tau(x)\leq 0,\;\;\forall x\in J(\hbar).
\]
Let us observe   that 
\begin{equation}
\lim_{\tau\to 0} \vfi_\tau(x)=\frac{1}{2}\vfi''_\tau(0)x^2=-\frac{x^2}{2},\;\;\forall x\in \bR.
\label{eq: limtau}
\end{equation}
Indeed, fix  $x\in \bR $ and assume $\tau$ is small enough so that 
\begin{equation}
\tau|x| <\frac{1}{2}.
\label{eq: small}
\end{equation}
Observe that
\[
\vfi_\tau^{(j)}(0)=\frac{1}{a^2\tau^2}\left(\frac{2(1-\tau)}{\tau}\frac{d^j}{dx^j}|_{x=0}\log\left(1-\frac{\si}{1-\tau} x\,\right)-\frac{d^j}{dx^j}|_{x=0}\left(\frac{1}{(1+\frac{\si}{\tau}x)^2}-1\right)\right)
\]
\[
=\frac{1}{a^2\tau^2}\left(-\frac{2(1-\tau)}{\tau}\left(\frac{\si}{1-\tau}\right)^j+(-1)^{j+1}(j+1)!\left(\frac{\si}{\tau}\right)^j\right).
\]
Using the estimate $\si=O(\tau^2)$ as $\tau\to 0$ we deduce that there exists $C>0$ such that, for any $j\geq 0$ we have
\[
\bigl|\vfi_\tau^{(j)}(0)|\leq C (j+1)! \tau^{j-2}.
\]
Hence
\[
\frac{1}{j!}\bigl|\vfi_\tau^{(j)}(0) x^j\,\bigr| \leq Cj|\tau x|^{j-2} x^2,\;\; \forall  j\geq 2.
\]
Thus   if $\tau$ satisfies (\ref{eq: small}), we have
\[
\vfi_\tau(x)+\frac{x^2}{2}=\vfi_\tau(x)-\vfi'_\tau(0)x-\frac{1}{2}\vfi''_\tau(0)x^2=\sum_{j\geq 3}\frac{1}{j!}\vfi_\tau^{(j)}(0)x^j,
\]
where the series in the right-hand side is absolutely convergent. Hence
\[
\bigl|\, \vfi_\tau(x)+\frac{x^2}{2}\,\bigr| \leq Cx^2|\tau x|\sum_{j\geq 3}j|\tau x|^{j-3}\leq C|\tau x| x^2 \sum_{j\geq 3}j2^{j-3}.
\]
This proves (\ref{eq: limtau}). 

Next we want to prove that there exists a constant  $A>0$ such that
\begin{equation}
\vfi_\tau(x)\leq A(1-|x|),\;\;\forall x\in J(\hbar),\;\;\forall \tau\ll1.
\label{eq: concave}
\end{equation}
We will achieve this by relying on the concavity of $\vfi_\tau$ over the interval $J(\hbar)$. The graph of $\vfi_\tau$ is situated below either of the  lines tangent to the graph at $x=\pm 1$. Thus
\[
\vfi_\tau(x)\leq \vfi_\tau(1)+\vfi'_\tau(1)(x-1)\leq - \vfi'_\tau(1)+\vfi'_\tau(1)x,
\]
\[
\vfi_\tau(x)\leq \vfi_\tau(-1)+\vfi'_\tau(-1)(x+1)\leq \vfi'_\tau(-1)+\vfi'_\tau(-1).
\]
Now observe that
\[
\frac{d}{dx}\vfi_\tau(x)=\frac{1}{a^2\tau^2}\left(-\frac{2\si}{\tau}\frac{1}{1-\frac{\si}{1-\tau}x} +\frac{2\si}{\tau}\frac{1}{(1+\frac{\si}{\tau}x)^3}\right)=\frac{2\si}{a^2\tau^3}\left(\frac{1}{(1+\frac{\si}{\tau}x)^3}-\frac{1}{1-\frac{\si}{1-\tau}x}\right).
\]
Using the fact that  $\si=O(\tau^2)$ we deduce from the above  equality that
\[
|\vfi'_\tau(\pm 1)|= O(1),\;\;\mbox{as $\tau\to 0$}.
\]
This proves  (\ref{eq: concave}). Using (\ref{eq: limtau}), (\ref{eq: concave}) and the dominated convergence theorem  we deduce
\[
\lim_{\tau\to \infty}\int_{J(\hbar)} e^{\vfi_\tau(x)} dx=\int_\bR e^{-\frac{x^2}{2}} dx=\sqrt{2\pi}.
\]
We conclude that
\begin{equation}
 P_\hbar\sim e^{\phi_\hbar(\tau)}a^{\frac{1}{\hbar}}\sqrt{2\pi v}\;\;\mbox{as $\hbar\to 0$}
 \label{eq: ph}
 \end{equation}
 Now observe that 
 \[
 \phi_\hbar(\tau) =\frac{1}{\hbar}\log(1-\tau) -\frac{1}{a^2\tau^2}=\frac{2(1-\tau)}{a^2\tau^2}\frac{\log(1-\tau)}{\tau}-\frac{1}{a^2\tau^2}\sim -\frac{3}{a^2\tau^2}.
 \]
 Using (\ref{eq: tauh}) we deduce
 \[
 \phi_\hbar(\tau)\sim -\frac{3}{a^2}\left(\frac{a^2}{2\hbar}\right)^{\frac{2}{3}}=-\frac{3}{(2a\hbar)^{\frac{2}{3}}}=-3\left(\frac{k}{2a}\right)^{\frac{2}{3}},\;\;k=\frac{1}{\hbar}.
 \]
 Also
 \[
 e^{\phi_\hbar(\tau)}=(1-\tau)^{\frac{2(1-\tau)}{a^2\tau^3}}e^{-\frac{1}{a^2\tau^2}}.
 \]
 In any case, using (\ref{eq: tauh}), (\ref{eq: vh}) and (\ref{eq: ph}) we deduce that
 \begin{equation}
 \log I_k(w) \sim  k\log a =k \log(c+1) \;\;\mbox{as $k\to\infty$}.
 \label{e: ik}
 \end{equation}
 Thus
 \[
 \log r_m(w)=\log\left(\frac{I_{m-1}(w) I_{m+3}(w)}{I_{m+1}(w)}\right)=0,
 \]
 so that
 \[
 \lim_{m\to\infty}q_m=\lim_{m\to\infty} r_m=1.\proofend
 \]
\end{ex} 

\begin{ex} If  we let $c=0$ in the above example , then we deduce that
\[
I_k(w)=\int_0^1t^k e^{-\frac{1}{1-t^2}}  dt \sim e^{\phi_\hbar(\tau)}\sqrt{2\pi v(\hbar)}
\]
where 
\[
\phi_\hbar(\tau)\sim -3\left(\frac{k}{2}\right)^{\frac{2}{3}},\;\;\;v(\hbar)\sim\frac{1}{6}\left(\frac{2}{k^2}\right)^{\frac{2}{3}}.
\]
Hence
\[
\log I_k(w)\sim-3\left(\frac{k}{2}\right)^{\frac{2}{3}},
\]
\[
\log r_m(w)\sim -\frac{3}{2^{\frac{2}{3}}}\Bigl( (m-1)^{\frac{2}{3}}+(m+3)^{\frac{2}{3}}-(m+1)^{\frac{2}{3}}\Bigr)\to 0,
\]
so that 
\[
 \lim_{m\to\infty}q_m=\lim_{m\to\infty} r_m=1.\proofend
 \]
\end{ex}
\section{A probabilistic proof of the Gauss-Bonnet theorem}
\setcounter{equation}{0}
\label{s: 4}

Suppose that  $M$ is a smooth, compact, connected  \emph{oriented} manifold  of even   dimension $m$. For any Riemann metric $g$ we can view the Riemann curvature tensor $R_g$  as a symmetric   bundle morphism $R_g:\Lambda^2TM\to \Lambda^2TM$. Equivalently, using the metric identification $T^*M\cong TM$ we can view $R_g$ as a section of $\Lambda^2 T^*M\otimes  \Lambda^2 T^*M$.

We will denote by $\Omega^{p,q}(M)$ the sections of $\Lambda^p T^*M\otimes \Lambda^q T^*M$ and we will refer to them of \emph{double forms} of type $(p,q)$. Thus $R_g\in \Omega^{2,2}(M)$.
We have a natural   product
\[
\bullet :\Omega^{p,q}(M)\times \Omega^{p',q'}(M)\to \Omega^{p+p',q+q'}(M)
\]
defined in   a natural way; see \cite[Eq. (7.2.3)]{AT} for  a precise definition. 

Using the   metric $g$  we can identify a double-form in $\Omega^{k,k}(M)$ with a section  of $\Lambda^kT^*M\otimes \Lambda^k TM$, i.e., with a  bundle morphism $\Lambda^k TM\to \Lambda^k TM$ and thus we have  a  linear map
\[
\tr:\Omega^{k,k}(M)\to  C^\infty(M).
\]
For $1\leq k\leq \frac{m}{2}$ we have a double form
\[
R_g^{\bullet k} =\underbrace{R_g\bullet\cdots \bullet R_g}_{k}\in \Omega^{2k,2k}(M).
\]
We denote by $dV_g\in \Omega^m(M)$ the volume \emph{form}  on $M$ defined by the metric $g$ and the orientation on $M$. We set
\[
\be_g(M):=\frac{1}{(2\pi)^{\frac{m}{2}} (\frac{m}{2})!} \tr\Bigl(-R_g^{\bullet \frac{m}{2}}\Bigr) dV_g\in \Omega^m(M).
\]
The form  $\be_g(M)$ is called the \emph{Euler form} of the metric   $g$ and the classical Gauss-Bonnet theorem states that
\begin{equation}
\int_M \be_g(M)=\chi(M)=:\mbox{the Euler characteristic of $M$}.
\label{eq:  GB}
\end{equation}
In this section we will show that  the Gauss-Bonnet theorem for any metric $g$  is an immediate consequence of  the Kac-Rice formula  coupled with the approximation theorem  Thm. \ref{th: metric}.

Fix a metric $g$. For simplicity we assume that ${\rm vol}_g(M)=1$.  This does not affect the generality since  $\be_{cg}(M)= \be_g(M)$ for any constant $c>0$.  Consider  the random function $\bu^\ve$ on $M$ defined by  (\ref{eq: uve}, \ref{eq: uve1}). Set
\[
\bv^\ve=\left(\frac{\ve^{m+2}}{d_m}\right)^{\frac{1}{2}}\bu^\ve.
\]
Observe that for $\ve>0$   sufficiently small, any $X,Y\in\Vect(M)$  and any $\bp\in M$ we have
\[
h^\ve(X(\bp), Y(\bp))=\bsE\bigl(\,X\bv^\ve(\bp),Y\bv^\ve(\bp)\,\bigr)
\]
where $h^\ve$ is the metric on $M$  that appears in the approximation theorem, Theorem \ref{th: metric}. 

For any    smooth function $f: M\to\bR$ and any  $\bp\in M$ we denote by $\Hess^\ve_\bp(f)$ the Hessian of $f$ at $\bp$ defined in terms  of the metric $h^\ve$. More precisely 
\[
\Hess_\bp^\ve(f)=  XY f(\bp)-(\nabla_X^\ve Y)f(\bp),\;\;\forall X,Y\in \Vect(M),
\]
where $\nabla^\ve$ denotes the Levi-Civita connection of the metric $h^\ve$. Using the metric $h^\ve$ we can identify this Hessian with a  symmetric   linear operator
\[
\Hess^\ve_\bp(f): (T_\bp M, h^\ve)\to (T_\bp M,h^\ve).
\]
For any  $\bp\in M$ we have a random vector $d\bv^\ve(\bp)\in T^*_\bp M$. Its covariance form $S(d\bv^\ve(\bp))$ is  precisely the metric $h^\ve$, and if we use the metric $h^\ve$ to identify this form with an operator we deduce that $S(d\bv^\ve(\bp))$ is identified with the identity operator.

For every   smooth Morse function  $f$ on $M$ and any integer  $0\leq k\leq m$  we have a measure $\nu_{f,k}$ on $M$
\[
\nu_{f,k}=\sum_{d f(\bp)=0,\;\ind(f,\bp)=k} \delta_\bp,
\]
where $\ind(f,\bp)$ denotes the \emph{Morse index} of the critical  point $\bp$ of the Morse function $f$. We set
\[
\nu_f =\sum_{k=0}^m (-1)^k \nu_{f,k}
\]
The  Poincar\'{e}-Hopf theorem implies that for any  Morse function  we have
\begin{equation}
\int_M \nu_f(d\bp)=\chi(M).
\label{eq: PH}
\end{equation}
Using the random Morse function $\bv^\ve$ we obtain the random measures $\nu_{\bv^\ve,\bp}$, $\nu_{\bv^\ve}$. We denote by  $\nu^\ve_k$  and respectively $\nu^\ve$  their  expectations. The Kac-Rice formula   implies that
\[
\nu_k =\frac{1}{(2\pi)^{\frac{m}{2}}} \rho^\ve_k(\bp) |dV_{h^\ve}(\bp)|,
\]
where 
\[
\rho_k^\ve(\bp)=\frac{1}{\sqrt{\det S(\bv^\ve(\bp))}} \bsE\Bigl(\, |\det \Hess^\ve_\bp(\bv^\ve)|\,\bigl|\,  d\bv^\ve(\bp)=0,\;\;\ind\Hess^\ve_\bp(\bv^\ve)=k\,\Bigr)
\]
\[
=(-1)^k  \bsE\Bigl(\, \det \Hess^\ve_\bp(\bv^\ve)\,\bigl|\,  d\bv^\ve(\bp)=0,\;\;\ind\Hess^\ve_\bp(\bv^\ve)=k\,\Bigr).
\]
As shown in \cite[Eq. (12. 2.11)]{AT}, the Gaussian  random variables $\Hess_\bp^\ve(\bv^\ve)$ and $d\bv^\ve(\bp)$ are independent so that
\[
\rho_k^\ve(\bp)=(-1)^k  \bsE\Bigl(\, \det \Hess^\ve_\bp(\bv^\ve)\;\bigl|\; \ind\Hess^\ve_\bp(\bv^\ve)=k\,\Bigr).
\]
Thus
\[
\nu^\ve= \frac{1}{(2\pi)^{\frac{m}{2}}}\sum_{k=0}^m(-1)^k \rho_k^\ve(\bp) |dV_{h^\ve}(\bp)|,
\]
\[
=\frac{1}{(2\pi)^{\frac{m}{2}}}\sum_{k=0}^m  \bsE\Bigl(\, \det \Hess^\ve_\bp(\bv^\ve)\;\bigl|\; \ind\Hess^\ve_\bp(\bv^\ve)=k\,\Bigr)|dV_{h^\ve}(\bp)|
\]
\[
=\frac{1}{(2\pi)^{\frac{m}{2}}}\bsE\Bigl(\, \det \Hess^\ve_\bp(\bv^\ve)\,\Bigr) |dV_{h^\ve}(\bp)|.
\]
From the Poincar\'{e}-Hopf equality (\ref{eq: PH}) we deduce
\begin{equation}
\chi(M)=\int_M \nu^\ve(d\bp)= \frac{1}{(2\pi)^{\frac{m}{2}}}\int_M\bsE\Bigl(\, \det \Hess^\ve_\bp(\bv^\ve)\,\Bigr) |dV_{h^\ve}(\bp)|.
\label{eq: sto_HP}
\end{equation}
Observe that   Hessian  $\Hess^\ve(f)$  of a function $f$ can also be viewed  as a  double form 
\[
\Hess^\ve(f)\in\Omega^{1,1}(M).
\]
In particular, $\Hess^\ve(\bv^\ve)$ is a random $(1,1)$ double form and we have the following equality, \cite[Lemma 12.2.1]{AT}
\begin{equation}
-2R_{h^\ve} =\bsE\bigl( \Hess^\ve(\bv^\ve)^{\bullet 2}\,\bigr),
\label{eq: rhess}
\end{equation}
where $R_{h^\ve}$ denotes the Riemann curvature tensor of the metric $h^\ve$. On the other hand we have the  equality \cite[Eq. (12.3.1)]{AT}
\begin{equation}
\det \Hess^\ve(\bv^\ve) =\frac{1}{m!} \tr \Hess^\ve(\bv^\ve)^{\bullet m}
\label{eq: det-hess}
\end{equation}
Using (\ref{eq: rhess}), (\ref{eq: det-hess}) and the  algebraic identities in \cite[Lemma 12.3.1]{AT} we conclude that
\[
\frac{1}{(2\pi)^{\frac{m}{2}}}\bsE\Bigl(\, \det \Hess^\ve_\bp(\bv^\ve)\,\Bigr) =\frac{1}{(2\pi)^{\frac{m}{2}} (\frac{m}{2})!} \tr\Bigl(-R_{h^\ve}^{\bullet \frac{m}{2}}\Bigr) .
\]
This proves (\ref{eq: prob-GB}). Using  this equality in  (\ref{eq: sto_HP}) we deduce
\[
\chi(M)=\int_M \be_{h^\ve}(M),
\]
i.e., we have proved the Gauss-Bonnet theorem for the  metric  $h^\ve$. Now let $\ve\to 0$.  As we have mentioned, Theorem \ref{th: metric} implies that $h^\ve\to g$  so in the limit, the above equality reduced to the Gauss-Bonnet theorem for the  original metric $g$.

\appendix

\section{Jets of the distance function}
\label{s: dist}
\setcounter{equation}{0}
Suppose that $(M,g)$ is a smooth, $m$-dimensional manifold, $\bp_0\in M$,  $U$ is an open, geodesically convex neighborhood  of $\bp_0$ and $(x^1, \dotsc, x^m)$ are normal coordinates on $U$ centered at $\bp_0$.    We have  a smooth function
\[
\eta: U\times U\ra [0,\infty),\;\;  \eta(\bp,\bq)= {\rm dist}_g(\bp,\bq)^2.
\]
We  want to investigate the partial derivatives of $r$ at $(\bp_0,\bp_0)$. Using  the above  normal coordinates we regard  $\eta$ as a function $\eta=\eta(x,y)$ defined in an open neighborhood of $(0,0)\in\bR^m\times \bR^m$.

If $f=f(t^1,\dotsc, t^N)$  is a smooth  function defined in a neighborhood of $0\in\bR^N$ and $k$ is a nonnegative integer, then we denote by $[f]_k$ the degree $k$-homogeneous part in the Taylor expansion of $f$ at $0$, i.e.,
\[
[f]_k= \frac{1}{k!}\sum_{|\alpha|=k} \pa_t^\alpha f|_{t=0} t^\alpha\in\bR[t^1,\dotsc, t^N].
\]
In the coordinates $(x^i)$ the metric  $g$ has the form (using Einstein's summation convention throughout)
\[
g= g_{ij}dx^i dx^j,
\]
where $g_{ij}$  satisfy the estimates \cite[Cor. 9.8]{Gray}
\begin{equation}
g_{k\ell} =\delta_{k\ell} - \frac{1}{3}R_{ikj\ell}(0)x^ix^j +O(|x|^3).
\label{eq:  g1}
\end{equation}
We   deduce  that
\begin{equation}
g^{k\ell}=\delta_{k\ell} +\frac{1}{3} R_{ikj\ell}(0)x^ix^j +O(|x|^3).
\label{eq: g2}
\end{equation}
The function $\eta$ satisfies a Hamilton-Jacobi equation, \cite[p. 171]{Riesz},
\begin{equation}
g^{k\ell} \frac{\pa \eta(x,y)}{\pa x^k} \frac{\pa \eta(x,y)}{\pa x^\ell} =4 \eta (x,y),\;\;\forall x,y.
\label{eq: HJ}
\end{equation}
Moreover, $\eta$ satisfies the obvious symmetry conditions
\begin{equation}
\eta(x,y)=\eta(y,x),\;\;\; \eta(0,x)=\eta(x,0)= |x|^2:=\sum_{i=1}^m(x^i)^2.
\label{eq: S}
\end{equation}
As shown in  \cite[Lemma 2.2]{Bin}  we have
\begin{equation}
[\eta]_2= |x-y|^2=\sum_{i=1}^m (x^i-y^i)^2.
\label{eq:  j2}
\end{equation}
The symmetries (\ref{eq: S}) suggest the introduction of new coordinates $(u,v)$ on $U\times U$,
\[
u_i= x^i-y^i,\;\;v_j=  x^j+y^j.
\]
Then
\[
x^i=\frac{1}{2}(u_i+v_i),\;\;y^j=\frac{1}{2}(v_j-u_j),\;\; \pa_{x^i} =\pa_{u_i}+\pa_{v_i}.
\]
The  equality (\ref{eq: g2}) can be rewritten as
\begin{equation}
g^{k\ell}(x)=\delta^{k\ell}+\frac{1}{12}\sum_{i,j} R_{ikj\ell}(u_i+v_i)(u_j+v_j) + O(3).
\label{eq:  g3}
\end{equation}
The symmetry   relations (\ref{eq: S}) become
\begin{equation}
\eta(u,v)=\eta(-u,v),\;\; \eta(u,u)=|u|^2,
\label{eq: S1}
\end{equation}
while (\ref{eq: j2})  changes to
\begin{equation}
[\eta]_1=0,\;\;[\eta]_2= |u|^2.
\label{eq: r2}
\end{equation}
The equality (\ref{eq: HJ}) can be rewritten
\begin{equation}
\sum_{k,l} g^{kl}(x) \underbrace{\bigl(\, \eta'_{u_k} +\eta'_{v_k}\,\bigr)}_{=:A_k}\underbrace{\bigl(\, \eta'_{u_\ell}+ \eta'_{v_\ell}\,\bigr)}_{=:A_\ell}= 4 \eta.
\label{eq: HJ1}
\end{equation}
Note that
\begin{equation}
[A_k]_0=[A_\ell]_0=[g^{k,\ell}]_1=0,
\label{eq: a0}
\end{equation}
while (\ref{eq: r2}) implies that
\[
[A_k]_1=2u^k.
\]
We deduce
\[
4[\eta]_3=\sum_{k,\ell}[g^{kl}]_0\bigl(\, [A_k]_1[A_\ell]_2+ [A_k]_2[B_\ell]_1\,\bigr) =\sum_k  2[A_k]_2[A_k]_1=4\sum_k u_k[A_k]_2.
\]
We can rewrite this last equality as a differential equation for $[\eta]_3$ namely
\[
[\eta]_3=\sum_k u_k(\pa_{u_k}+\pa_{v_k}) [\eta]_3.
\]
We set $P=[\eta]_3$ so that $P$ is a homogeneous polynomial of degree $3$ in the variables  $u,v$. Moreover, according to (\ref{eq: S1})   the polynomial $P$ is even  in $u$ and $P(u,u)=0$. Thus $P$ has the form
\[
P=\underbrace{\sum_i C_i(u) v_i}_{=:P_2} \;+ \; P_0(v),
\]
where $C_i(u)$ is a homogeneous polynomial of degree $2$ in the variables $u$,  and $P_0(v)$ is homogeneous of degree $3$ in the variables $v$.

We have
\[
\sum_k u_k\pa_{v_k} P_2= \underbrace{\sum_k C_k(u) u_k}_{=: Q_3},\;\; Q_1:=\sum_k u_k\pa_{v_k} P_0,\;\;\sum_k u_k\pa_{u_k}P_0=0,
\]
and  the classical Euler equations imply
\[
\sum_ku_k\pa_{u_k} P_2= 2P_2.
\]
We deduce
\[
P= 2P_2+ Q_3+Q_1,
\]
where the polynomials $Q_3$ and $Q_1$ are odd in the variable $u$. Since $P$ is even in the variable $u$ we deduce
\[
Q_3+Q_1=0,
\]
so that $P_2+P_0=P=2P_2$. Hence $P_2=P_0=0$ and  thus
\begin{equation}
[\eta]_3=0.
\label{eq: r3}
\end{equation}
In particular
\begin{equation}
[A_k]_2=0,\;\;\forall k.
\label{eq: a2}
\end{equation}
Going back to (\ref{eq: HJ1}) and using (\ref{eq: a0}) and (\ref{eq: a2}) we deduce
\begin{equation}
\begin{split}
4[\eta]_4=  \sum_{k,\ell}[g^{k\ell}]_2[A_k]_1[A_\ell]_1 +\sum_{k,\ell}[g^{k\ell}]_0 \bigr([A_k]_1[A_\ell]_3+[A_k]_3[A_\ell]_1\bigl)\\
=4\sum_{k,\ell}[g^{k\ell}]_2u_ku_\ell+ 2\sum_k u_k[A_k]_3.
\end{split}
\label{eq: HJ2}
\end{equation}
We set $P=[\eta]_4$.  The polynomial $P$ is homogeneous of degree $4$ in the variables $u,v$, and it is even in the variable $u$. We can write $P= P_0+P_2+P_4$, where
\[
P_4=\sum_k c_{ijkl}u_iu_ju_ku_\ell,\;\;P_2=\sum_{i,j} Q_{ij}(u) v_iv_j,
\]
and  $P_0$ is homogeneous of degree $4$ in the variables $v$, $Q_{ij}(u)$ is a  homogeneous quadratic polynomial in the variables  $u$.  We have
\[
\sum_k u_k[A_k]_3=\sum_k u_k(\pa_{u_k}+\pa_{v_k})P.
\]
We have
\[
\sum_ku_k \pa_{u_k} P_{2\nu}= 2\nu P_{2\nu},\;\;\nu=0,1,2.
\]
\[
\sum_k u_k\pa_{v_k}P_4=0,
\]
\[
\sum_k u_k \pa_{v_k} P_2=\sum_{k,i,j} u_k Q_{ij}(\delta_{ki}v_j+\delta_{kj}v_i)=\sum_{k ,j}\bigl( Q_{kj}u_kv_j+Q_{jk}v_ju_k\,\bigr)
\]
Using these equalities in  (\ref{eq: HJ2}) we deduce
\[
4P_4+4P_2+4P_0 = 4\sum_{k,\ell}[g^{k\ell}]_2u_ku_\ell+ 4P_4+ 2P_2 + \sum_ku_k\pa_{v_k} P_0
\]
\[
+\sum_{k ,j}\Bigl( Q_{jk}+Q_{kj}\,\Bigr)u_kv_j.
\]
This implies   $P_0=0$ so that $P =P_4+P_2$, and we can then rewrite the above equality as
\begin{equation}
P_2= 2\sum_{k,\ell}[g^{k\ell}]_2u_ku_\ell+\sum_{k ,j}\Bigl( Q_{jk}+Q_{kj}\,\Bigr)u_kv_j. 
\label{eq: HJ4}
\end{equation}
Note that the   equality $r(u,u)=|u|^2$ implies $P(u,u)=0$ so that
\[
P_4(u)=P_4(u,u)=-P_2(u,u).
\]
There fore  it suffices to determine $P_2$. This can be achieved using the equality (\ref{eq: g3}) in (\ref{eq: HJ4}).  We have
\[
2\sum_{k,\ell}[g^{k\ell}]_2u_ku_\ell=\frac{1}{6}\sum_{i,j,k,\ell} R_{ikj\ell}(u_i+v_i)(u_j+v_j)u_ku_\ell
\]
\[
=\frac{1}{6}\sum_{i,j}\; \underbrace{\Bigl(\sum_{k,\ell} R_{ikj\ell}u_ku_\ell\, \Bigr) }_{\widehat{Q}_{ij}(u)}v_iv_j + \sum_j S_j(u)v_j,
\]
where $S_j(u)$ denotes a  homogeneous polynomial of degree $3$ in $u$.
The equality (\ref{eq: HJ4}) can now be rewritten as
\[
\sum_{i,j} Q_{ij}(u)v_iv_j =\frac{1}{6}\sum_{i,j}\widehat{Q}_{ij}(u) v_iv_j +\sum_j S_j(u)v_j+  \frac{1}{2}\sum_{k ,j}\Bigl( Q_{jk}+Q_{kj}\,\Bigr)u_kv_j.
\]
From this we read easily 
\[
Q_{ij}(u) =\frac{1}{6}\widehat{Q}_{ij}(u)=\frac{1}{6}\sum_{k,\ell} R_{ikj\ell}u_ku_\ell.
\]
This determines $P_2$. 
\begin{equation}
P_2(u,v)=\frac{1}{6}\sum_{i,j} \widehat{Q}_{ij}(u)v_iv_j.
\label{eq: p2}
\end{equation}
As we have indicated above $P_2$ determines $P_4$.
\begin{equation}
P_4(u)=-P_2(u,u)= -\frac{1}{6}\sum_{i,j,k,\ell} R_{ikj\ell}u_iu_ju_ku_\ell.
\label{eq: p4}
\end{equation}
The skew symmetries  of the Riemann tensor imply that $P_4=0$ so that
\begin{equation}
[\eta]_4(u,v) = \frac{1}{6}\sum_{i,j} \widehat{Q}_{ij}(u)v_iv_j,\;\;\widehat{Q}_{ij}(u)=\sum_{k,\ell} R_{ikj\ell}u_ku_\ell.
\label{eq: r4}
\end{equation}

\begin{ex} Suppose that $M$ is a surface, i.e., $m=2$.    Set
\[
K = R_{1212}= R_{2121}=-R_{1221}.
\]
Note that $K$ is the Gaussian curvature of the surface. Then
\[
\widehat{Q}_{11}=\sum_{k,\ell}R_{1k1\ell}u_ku_\ell= Ku_2^2,\;\;\widehat{Q}_{22}=\sum_{k,\ell}R_{2k2\ell}u_ku_\ell= Ku_1^2.
\]
\[
\widehat{Q}_{12}=\sum_{k,\ell} R_{1k2\ell}u_ku_\ell= -Ku_1u_2=\widehat{Q}_{21}.
\]
Hence
\[
P_2(u,v)=\frac{K}{6}(u_2^2v_1^2+u_1^2v_2^2- 2u_1u_2v_1v_2)=\frac{K}{6} (u_1v_2-u_2v_1)^2. 
\]
\qed
\end{ex}

\section{Spectral estimates}\label{s: spec}
\setcounter{equation}{0}

As we have already mentioned, the correlation function
\[
\eE^\ve(\bp,\bq)=\sum_{k\geq 0} w_\ve(\sqrt{\lambda_k})\Psi_k(\bp)\Psi_k(\bq)
\]
is the  Schwartz kernel  of the    smoothing operator $w_\ve(\sqrt{\Delta})$.     In  this appendix we present in some detail information about the behavior along the diagonal of this kernel as $\ve\to 0$.    We will achieve this by relying on the wave kernel technique pioneered by L. H\"{o}rmander, \cite{Hspec}.

The fact that such asymptotics exist and can be obtained in this fashion is well known to experts; see e.g \cite{DG} or  \cite[Chap.XII]{Tay}.  However,    we could not find any reference describing these asymptotics with the level of specificity   needed for the considerations in this paper.  

\begin{theorem} Suppose that $w\in \eS(\bR)$ is an even, nonnegative Schwartz function, and $(M,g)$ is a smooth, compact, connected $m$-dimensional Riemann manifold. We define
\[
\eE^\ve:M\times M\to \bR,\;\;\eE^\ve(\bp,\bq)=\sum_{k\geq 0} w(\ve\sqrt{\lambda_k})\Psi_k(\bp)\Psi_k(\bq),
\]
where $(\Psi_k)_{k\geq 1}$ is an orthonormal  basis of $L^2(M,g)$ consisting of   eigenfunctions of $\Delta_g$.

Fix a point $\bp_0\in M$ and normal coordinates  at $\bp_0$ defined  in an open neighborhood $\eO_0$ of $\bp_0$. The restriction of $\eE^\ve$ to $\eE^\ve$ to $\eO_0\times \eO_0$ can be viewed as a function $\eE^\ve(x,y)$ defined in an open neighborhood of $(0,0)$ in $\bR^m\times \bR^m$.   Fix multi-indices $\alpha,\beta\in (\bZ_{\geq 0})^m$. Then

\begin{equation}
\pa^\alpha_x\pa^\beta_y\eE^\ve(x,y)|_{x=y=0} = \ve^{-m-2d(\alpha,\beta)} \frac{\ii^{|\alpha|-|\beta|}}{(2\pi)^m}\left(\,\int_{\bR^m} w(|x|) x^{\alpha+\beta} dx+ O(\ve^2)\,\right),\;\;\ve\to 0,
\label{eq: key-asym}
\end{equation}
where
\[
d(\alpha,\beta):=\left\lfloor\frac{|\alpha+\beta|}{2}\right\rfloor.
\]
Moreover, the  constant  implied by the symbol  $O(\ve)$  in (\ref{eq: key-asym})  uniformly bounded with  respect to $\bp_0$.

\label{th: key-asym}
\end{theorem}

\begin{proof} For the reader's convenience and for later use,  we   go  in some detail through the process of obtaining these asymptotics. We skip many analytical steps that are well   covered in \cite[Chap. 17]{H3} or \cite{N3}.

  Observe that  for any  smooth $f:M\to\bR$ we have
\begin{equation}
w_\ve(\sqrt{\Delta}) f= \frac{1}{2\pi}\int_\bR  \widehat{w_\ve}(t) e^{\ii t\sqrt{\Delta}}f dt=\frac{1}{2\pi\ve}\int_\bR\widehat{w}\left(\frac{t}{\ve}\right) e^{\ii t\sqrt{\Delta}}f dt .
\label{eq: four}
\end{equation}
The Fourier transform $\widehat{w}(t)$ is a Schwartz function  so  $\widehat{w}(t/\ve)$ is really small for $t$ outside a small interval around $0$ and $\ve$ sufficiently small. Thus a good  understanding  of the kernel  of $e^{\ii t\sqrt{\Delta}}$  for $t$ sufficiently small  could potentially lead to a good understanding of the Schwartz kernel of  $w_\ve(\sqrt{\Delta})$. 

Fortunately, good    short time  asymptotics for the wave kernel are available.   We will describe  one such method going back to  Hadamard, \cite{Had, Riesz}. Our presentation follows closely \cite[\S 14.4]{H3} but we also refer to  \cite{N3} where we have   substantially expanded the often  dense presentation in \cite{H3}.

To  describe these asymptotics we need to introduce  some important families   homogeneous   generalized functions  (or distributions) on $\bR$.  We will denote by $C^{-\infty}(\Omega)$ the  space of generalized functions on  the smooth manifold $\Omega$, defined as the dual of the space   compactly supported $1$-densities, \cite[Chap. VI]{GS}.

For any  $a\in\bC$, $\re a>1$ we define  $\chi^a_+:\bR\to\bR$ by
\[
\chi_+^a(x) =\frac{1}{\Gamma(a+1)} x_+^a,\;\;x_+=\max(x,0).
\]
Observe that we have the following   equality in the sense of distributions
\[
\frac{d}{dx}\chi_+^{a+1}=\chi_+^a(x),\;\; \re a>1.
\]
We can use this to define for any $a\in \bC$
\[
\chi_+^a:=\frac{d^k}{dx^k}\chi^{a+k}_+\in C^{-\infty}(\bR),\;\;k>1-\re a .
\]
For $\re a>0$ we  denote by $|\chi|^a$ the   generalized function defined by the locally integrable function
\[
|\chi|^a(x) =\frac{1}{\Gamma(\frac{a+1}{2})} |x|^a.
\]
The     correspondence $a\mapsto |\chi|^a$    is a  holomorphic  map $\{\re z>0\}\to   C^{-\infty}(\bR)$ which admits a  holomorphic extension to the whole complex plane, \cite[Chap. 1]{GeVi}, \cite{N3}. This is a temperate generalized function, and its Fourier transform is given by, \cite{GeVi, N3},
\begin{equation} 
\widehat{|\chi|^a}\,(\xi)=\sqrt{\pi}2^{a+1} |\chi|^{-(a+1)}(\xi),\;\;\forall  a\in\bC.
\label{eq: ab-four}
\end{equation}

Denote by $K_t(x,y)$ the Schwartz kernel of $e^{\ii t\sqrt{\Delta}}$.    We then have the following result \cite[\S 17.4]{H3} or \cite{N3}.

\begin{theorem} Set $n:=m+1$,  and let
\[
\eta(x,y)=\dist_g(x,y)^2,\;\; x,y\in M.
\]
 There exists a positive  constant  $c>0$, smaller than the injectivity radius of $(M,g)$,  such that  for $\dist_g(x,y)<c$ we have the following asymptotic expansion as $t\to 0$

\begin{equation} 
K_t(\bp,\bq)\sim \sum_{k=1}^\infty U_k(\bp, \bq)d_m(2k)\eH_k(t,\bp,\bq),\;\;|t|< c,
\label{eq: wker}
\end{equation}
where for $\re a>0$ we have
\[
\eH_a(t,\bp,\bq)=\pa_t\left(\,\chi_+^{a-\frac{n}{2}}\bigl(t_+^2-\eta(\bp,\bq)\,\bigr) -\chi_+^{a-\frac{n}{2}}\bigl(t_-^2-\eta(\bp,\bq)\,\bigr)\,\right),
\]
\[
d_m(2a)=\frac{\Gamma(\frac{2a+1}{2})}{\pi^{\frac{m}{2}} \Gamma(2a)}.
\]
\label{th: fund}
\end{theorem}

Let us  explain in  more detail the meaning of the above result.  The   functions  $U_k$   are   smooth functions defined  in   the neighborhood $\dist_g(\bp,\bq)< c$ of the diagonal in $M\times  M$. For fixed  $\bq$, the functions $\bp \mapsto V_k(\bp):=U_k(\bp,\bq)$ are  determined as follows.

Fix  normal coordinates $x$ at $\bq$, set $|g|:=\det(g_{ij})$, and
\[
h(x): = -\frac{1}{2}g\bigl(\,\nabla\log |g|,\;  x\,\bigr)=-\frac{1}{2}\sum_{j,k}g^{jk}x^j\pa_{x^k}\log|g|.
\]
Then $V_k(x)$ are the unique solutions  of the   differential recurrences 
\begin{equation}
V_1(0)=1,\;\;2x\cdot\nabla V_1= hV_1,\;\;|x|< c,
\label{eq: Ua}
\end{equation}
\begin{equation}
\frac{1}{k}x\cdot \nabla V_{k+1}+ \left(1-\frac{1}{2k} h\right) V_{k+1}=-\Delta_g V_k,\;\;V_{k+1}(0)=0,\;\;|x|< c,\;\; k\geq 1.
\label{eq: Ub}
\end{equation}

We have the  following important  equality
\begin{equation}
 \lim_{\dist_g(\bp,\bq)\to 0}\eH_a(t,\bp,\bq)= |\chi|^{2a-2-m}(t),\;\;\forall a\in\bC.
  \label{eq: prope}
  \end{equation}

The  asymptotic estimate (\ref{eq: wker}) signifies  that for any positive integer $\mu$  there exists a positive integer $N(\mu)$ so that for any $N\geq N(\mu)$ the tail
\[
\widetilde{\eT}_N(t,\bp,\bq):=K_t(\bp,\bq) - \sum_{k=1}^N U_k(\bp, \bq)d_m(2k)\eH_k(t,\bp,\bq)
\]
belongs to $C^\mu\bigl(\, (-c,c)\times M\times M\,\bigr)$ and satisfies the estimates 
\begin{equation}
\bigl\| \pa^j_t\, \widetilde{\eT}_N(t,-,-)\,\bigr\|_{C^{\mu-j}(M\times M)}  \leq  C|t|^{2N-n-1-\mu},\;\;|t|\leq c,\;\;,j\leq \mu,\;\;N\geq N(\mu).
\label{eq: tail4}
\end{equation}

Fix a point $\bp_0\in M$ and normal coordinates at $\bp_0$ defined in a neighborhood $\eO_0$ of $\bp_0$.   Then we can identify a point $(\bp,\bq)\in \eO_0\times \eO_0$ with  a point $(x,y)$ in a neighborhood of $(0,0)$ in $\bR^m\times\bR^m$.

Using (\ref{eq: four}) we deduce
 \begin{equation}
\pa^\alpha_x\pa^\beta_y\eE^{\ve}(x,y)|_{x=y}=\frac{1}{\ve} \Bigl\lan\, \underbrace{\pa^\alpha_x\pa^\beta_yK_t(x,y)|_{x=y}}_{=:K^{\alpha,\beta}_t}, \widehat{w}\left(\frac{t}{\ve}\right)\,\Bigr\ran.
\label{eq: wave-ker0}
\end{equation}
Choose an even, nonnegative    cutoff function $\rho\in C_0^\infty(\bR)$ such that
\[
\rho(t)=\begin{cases}
1, & |t|\leq \frac{c}{4},\\
0, & |t|\geq \frac{c}{2},
\end{cases}
\]
where $c>0$ is the constant in Theorem \ref{th: fund}.   Then
\[
\pa^\alpha_x\pa^\beta_y\eE^{\ve}(x,y)|_{x=y}=\frac{1}{\ve} \Bigl\lan K^{\alpha,\beta}_t ,\rho(t)\widehat{w}\left(\frac{t}{\ve}\right)\,\Bigr\ran+ \frac{1}{\ve} \Bigl\lan K^{\alpha,\beta}_t ,\bigl(\,1-\rho(t)\,\bigr)\widehat{w}\left(\frac{t}{\ve}\right)\,\Bigr\ran.
\]
Let us observe that that for any $N>0$
\[
\frac{1}{\ve} \Bigl\lan K^{\alpha,\beta}_t ,\bigl(\,1-\rho(t)\,\bigr)\widehat{w}\left(\frac{t}{\ve}\right)\,\Bigr\ran= O(\ve^N)\;\;\mbox{as $\ve\to 0$}
\]
Thus
\begin{equation}
\forall N>0\;\;\pa^\alpha_x\pa^\beta_y\eE^{\ve}(x,y)|_{x=y}\sim \frac{1}{\ve} \Bigl\lan K^{\alpha,\beta}_t ,\rho(t)\widehat{w}\left(\frac{t}{\ve}\right)\,\Bigr\ran+ O(\ve^N),\;\;\mbox{$\ve\to 0$}.
\label{eq: wave-ker0a}
\end{equation}
On the other hand
\begin{equation}
\pa^\alpha_x\pa^\beta_y K_t(x,y)\sim  \sum_{k=1}^\infty d_m(2k)\pa^\alpha_x\pa^\beta_y\left\{\,U_k(x, y)\eH_k(t,x,y)\,\right\}.
\label{eq: wave-ker1}
\end{equation}
Recall that
\[
d(\alpha,\beta)=\left\lfloor\frac{1}{2}|\alpha+\beta|\right\rfloor.
\]
One can show (see \cite{Bin,N3})
\begin{equation}
\pa^\alpha_x\pa^\beta_yK_t(x,y)|_{x=y=0}\sim \sum_{k=0}^\infty A_{m,\alpha,\beta, k} |\chi|^{-m-2d(\alpha,\beta)+2k}(t),
\label{eq: wave-ker2}
\end{equation}
where $A_{m,\alpha,\beta, 0}$  is a universal constant depending \emph{only} on $m,\alpha,\beta$, which is equal to $0$ if $|\alpha+\beta|$ is odd.

\begin{lemma} (a) For any  $r\in\bZ$  and any $N>0$ we have
\[
\frac{1}{\ve}\bigl\lan |\chi|^r , \rho\widehat{w_\ve}\,\bigr\ran=\ve^r\Bigl(\,\bigl\lan|\chi|^r, \widehat{w}\bigr\ran + O(\ve^N) \,\Bigr)\;\;\mbox{as $\ve \to 0$}.
\]
(b) For every positive integer $r$ we have
\[
\lan |\chi|^{-r}, \widehat{w}\ran =\frac{\sqrt{\pi}2^{1-r}}{\Gamma(\frac{r}{2})}\int_\bR|\tau|^{r-1} w(\tau) d\tau.
\]
\label{lemma: abs}
\end{lemma}

\begin{proof}     (a)  For transparency we will use the integral notation for the pairing between a generalized function and a test function. We have
\[
\bigl\lan |\chi|^r , \eta\widehat{w_\ve}\,\bigr\ran=\frac{1}{\ve}\int_\bR |\chi|^r(t) \rho(t)\widehat{w}(t/\ve) dt = \int_\bR|\chi|^r(\ve t)\rho(\ve t)\widehat{w}(t) dt
\]
\[
=\ve^r  \int_\bR|\chi|^r( t)\rho(\ve t)\widehat{w}(t) dt=\ve^r\lan |\chi|^r, \rho_\ve\widehat{w}\ran,\;\;\rho_\ve(t)=\rho(\ve t).
\]
Now observe that  $\rho_\ve \widehat{w}-\widehat{w}= \widehat{w}(\rho_\ve -1)\to 0$ in $\eS(\bR)$. More precisely for $k\geq 0$ we have
\[
\frac{\pa^k}{\pa k} (\rho_\ve -1)= O(\ve^Nt^N)\;\; \mbox{as $\ve \to 0$}.
\]
This implies that
\[
\bigr\lan |\chi|^r, \widehat{w}(\rho_\ve -1)\,\bigr\ran= O(\ve^N)\;\; \mbox{as $\ve \to 0$}, 
\]
so that 
\[
\lan |\chi|^r, \rho_\ve\widehat{w}\ran= \lan |\chi|^r, \widehat{w}\ran+ \bigr\lan |\chi|^r, \widehat{w}(\rho_\ve -1)\,\bigr\ran = \lan |\chi|^r, \widehat{w}\ran+ O(\ve^N)\;\; \mbox{as $\ve \to 0$}.
\]
(b) We have 
\[
\lan |\chi|^{-r}, \widehat{w}\ran= \bigl\lan\, \widehat{|\chi|^{-r}}, w\,\bigr\ran \stackrel{(\ref{eq: ab-four})}{=}\sqrt{\pi}2^{1-r}\bigl\lan |\chi|^{r-1}(\tau), w(\tau)\bigr\ran
\]
\[
=\frac{\sqrt{\pi}2^{1-r}}{\Gamma(\frac{r}{2})}\int_\bR|\tau|^{r-1} w(\tau) d\tau.
\]
\end{proof}
Using (\ref{eq: wave-ker0a}) and the above lemma we deduce
\begin{equation}
\pa^\alpha_x\pa^\beta_y\eE^\ve(x,y)|_{x=y}= D_{m,\alpha,\beta}\ve ^{-m-2d(\alpha,\beta)} +O\left(\ve ^{-m-2d(\alpha,\beta)+2} \right)\;\;\mbox{as $\ve \to 0$},
\label{eq: wave-ker3}
\end{equation}
where $D_{m,\alpha,\beta}$ is a universal constant that depends only on $m,\alpha,\beta$  which is $=0$ if $|\alpha+\beta|$ is odd,
\begin{equation}
D_{m,\alpha,\beta}= A_{m,\alpha,\beta,0}\frac{\sqrt{\pi}2^{1-r}}{\Gamma(\frac{r}{2})}\int_\bR|\tau|^{r-1} w(\tau) d\tau,\;\;r=m+2d(\alpha,\beta).
\label{eq: Am}
\end{equation}

To determine the constant $D_{m,\alpha,\beta}$ it suffices to compute it for  one particular $m$-dimensional Riemann manifold.   Assume that $(M,g)$ is the torus $T^m$ equipped with the flat metric
\[
g=\sum_{i=1}^m (d\theta^i)^2,\;\; 0\leq \theta^i\leq 2\pi.
\]
The eigenvalues   of  the  corresponding Laplacian $\Delta_m$ are 
\[
|\vec{k}|^2, \;\;\vec{k}=(k_1,\dotsc, k_m)\in\bZ^m. 
\]
 Denote by $\prec$ the lexicographic order on $\bZ^m$.  For  $\vec{\theta}=(\theta^1,\dotsc,\theta^m)\in\bR$  and $\vec{k}\in\bZ^m$ we set
\[
\lan\vec{k},\vec{\theta}\ran:=\sum_{j=1}^m k_j\theta^j.
\]
A real  orthornormal basis of $L^2(\bT^m)$ is given by the  functions
\[
\Psi_{\vec{k}}(\vec{\theta}) =\frac{1}{(2\pi)^{\frac{m}{2}}}\begin{cases}
1, &\vec{k}=\vec{0}\\
2^{\frac{1}{2}}\sin \lan\vec{k},\vec{\theta}\ran, & \vec{k}\succ \vec{0},\\
2^{\frac{1}{2}}\cos \lan\vec{k},\vec{\theta}\ran, & \vec{k} \prec \vec{0}.
\end{cases}
\]
Then
\[
\eE^\ve(\vec{\theta},\vec{\vfi}) =\frac{1}{(2\pi)^m}\sum_{\vec{k}\in\bZ^m}w(\ve|\vec{k}|) e^{\ii\lan\vec{k}, \vec{\theta}-\vec{\vfi}\ran},
\]
 so that
 \[
 \pa^{\alpha}_{\vec{\theta}}\pa^\beta_{\vec{\vfi}}\eE^\ve(\vec{\theta},0)=\frac{\ii^{|\alpha|-|\beta|}}{(2\pi)^m}\sum_{\vec{k}\in\bZ^m}w_\ve(|\vec{k}|)\vec{k}^{\alpha+\beta} e^{\ii\lan\vec{k}, \vec{\theta}\ran}.
 \]
  Define
 \[
W_m, u_\ve:\bR^m\to\bR,\;\;W_m(x)=w(|x|),\;\;u_\ve(x)=W_m(\ve x)x^{\alpha+\beta}.
 \]
 Using the Poisson summation formula \cite[\S 7.2]{H1}  we deduce
 \[
 \pa^{\alpha}_{\vec{\theta}}\pa^\beta_{\vec{\vfi}}\eE^\ve(0,0)= \frac{\ii^{|\alpha|-|\beta|}}{(2\pi)^m}\sum_{\vec{\nu}\in\bZ^m}\widehat{u}_\ve(2\pi\vec{\nu}).
 \]
 Observe that
 \[
 \widehat{u}_\ve(\xi)= \int_{\bR^m}e^{-\ii\lan\xi,x\ran} w(\ve |x|) x^{\alpha+\beta}  dx = (\ii\pa_\xi)^{\alpha+\beta}\left( \int_{\bR^m}e^{-\ii\lan\xi,x\ran}W_m(\ve x)dx\right)
 \]
 \[
 =\ve^{-m}(\ii\pa_\xi)^{\alpha+\beta} \left( \int_{\bR^m}e^{-\ii\lan\frac{1}{\ve}\xi,y\ran}W_m(y)dy\right)= \ve^{-m}(\ii\pa_\xi)^{\alpha+\beta} \widehat{W}_m\left(\frac{1}{\ve}\xi\right).
 \]
 Hence
 \[
  \pa^{\alpha}_{\vec{\theta}}\pa^\beta_{\vec{\vfi}}\eE^\ve(\vec{\theta},0)=\frac{\ii^{|\alpha|-|\beta|}}{(2\pi\ve)^m} 
 \sum_{\vec{\nu}\in\bZ^m}\left\{(\ii\pa_\xi)^{\alpha+\beta}\widehat{W}_m\left(\frac{1}{\ve}\xi\right)\right\}_{\xi=2\pi\vec{\nu}}.
 \]
 As $\ve\to 0$ we have
 \[
 \pa^{\alpha}_{\vec{\theta}}\pa^\beta_{\vec{\vfi}}\eE^\ve(0,0)=\ve^{-m-|\alpha+\beta|} \frac{\ii^{|\alpha|-|\beta|}}{(2\pi)^m}\left (\,(\ii\pa_\xi)^{\alpha+\beta}\widehat{W}_m(0)+ O(\ve^N)\,\right),\;\;\forall N.
 \]
 Now observe that
 \[
  (\ii\pa_\xi)^{\alpha+\beta}\widehat{W}_m(0)= \int_{\bR^m}w(|x|)x^{\alpha+\beta} dx.
  \]
  so that
 \begin{equation}
  \pa^{\alpha}_{\vec{\theta}}\pa^\beta_{\vec{\vfi}}\eE^\ve(0,0)=\ve^{-m-|\alpha+\beta|} \frac{\ii^{|\alpha|-|\beta|}}{(2\pi)^m}\left(\,\int_{\bR^m} w(|x|) x^{\alpha+\beta} dx+ O(\ve^N)\,\right),\;\;\forall N.
  \label{eq: const1}
 \end{equation}
 This shows that
 \begin{equation}
 D_{m,\alpha,\beta}=\frac{\ii^{|\alpha|-|\beta|}}{(2\pi)^m}\int_{\bR^m} w(|x|) x^{\alpha+\beta} dx.
 \label{eq: Am1}
 \end{equation}
 This completes the proof of  Theorem \ref{th: key-asym}.
 \end{proof}
 
 \begin{remark}
  Note that 
  \[
 \int_{\bR^m}w(|x|)x^{\alpha+\beta} dx =\left(\int_{|x|=1} x^{\alpha+\beta} dA(x)\right)\underbrace{\left(\int_0^\infty w(r) r^{m+|\alpha+\beta|-1} dr\right)}_{=:I_{m,\alpha,\beta}(w)}.
 \]
 On the other hand, according to \cite[Lemma 9.3.10]{N1} we have
 
\begin{equation}
\int_{|x|=1} x^{\alpha+\beta} dA(x)=Z_{m,\alpha,\beta}:=\begin{cases}
\frac{2\prod_{i=1}^k\Gamma(\frac{\alpha_i+\beta_i+1}{2})}{\Gamma(\frac{m+|\alpha+\beta|}{2})}, &\alpha+\beta\in (2\bZ_{\geq 0})^m,\\
0, & {\rm otherwise.}
\end{cases}
\label{eq: zmab}
\end{equation}
We can now rewrite (\ref{eq: Am1}) as
\begin{equation}
  D_{m,\alpha,\beta}=\ve^{-m-|\alpha+\beta|} \frac{\ii^{|\alpha|-|\beta|}Z_{m,\alpha,\beta}}{(2\pi)^m}I_{m,\alpha,\beta}(w).
  \label{eq: const2}
  \end{equation}\qed
 \end{remark}
 
 \begin{theorem} Fix a point  $\bp\in M$ and normal coordinates   $(x^i)$ near $\bp$.  For $i\neq j$ we denote by $K_{ij}(\bp)$ the sectional curvature of $g$ at $\bp$ along  the plane spanned by $\pa_{x^i},\pa_{x^j}$. For  any multi-induces $\alpha,\beta\in (\bZ_{\geq 0})^m$ we set
 \[
 \eE^\ve_{\alpha;\beta}:=\pa^\alpha_x\pa^\beta_y\eE^\ve(x,y)|_{x=y=0}.
 \] 
 Then there exists a universal constant $\eZ_m$ that depends only on the dimension of $M$ and the weight $w$ such that
 \begin{equation}
 \eE^\ve_{ii;jj}-\eE^\ve_{ij;ij} = \eZ_mK_{ij}(\bp) \ve^{-m-2}\bigl(1+O(\ve^2)\,\bigr)\;\;\mbox{as $\ve\to 0$}.
 \label{eq: two-step}
 \end{equation}
 \label{prop: two-step}
 \end{theorem}

 \begin{proof} Using (\ref{eq: wave-ker2}) we deduce
 \begin{equation}
\eE^\ve_{ii;jj}-\eE^\ve_{ij;ij} \sim\frac{1}{\ve} \Bigl\lan K^{ii,jj}_t-K^{ij,ij}_t ,\eta(t)\widehat{w}\left(\frac{t}{\ve}\right)\,\Bigr\ran+ O(\ve^N),\;\;\mbox{$\ve\to 0$} 
\label{eq: wave-ker4}
\end{equation}
 On the other hand from  (\ref{eq: wave-ker0}) we conclue
\begin{equation}
K^{ii,jj}_t-K^{ij,ij}_t  \sim  \sum_{k=1}^\infty d_m(2k)\left(\pa^2_{x^i}\pa^2_{y^j}-\pa^2_{x^ix^j}\pa^2_{y^iy^j}\right)\left\{\,U_k(x, y)\eH_k(t,x,y)\,\right\}|_{x=y=0}
\label{eq: wve-ker5}
\end{equation}
To investigate the   above asymptotics we use  the technology in \cite{N3}.

Let us introduce some notations.  For a positive integer $k$ we denote by $\pa^k$ a generic  mixed-partial derivative of order $k$ in the variables $x^i$, $y^j$.  We denote  by $\pa^k\eta$ the  collection of  $k$-th order derivatives of $\eta(x,y)$.  $\eP_i(X)$ will denote   a homogeneous polynomial of degree $i$ in the variables $X$, while $\eP_k(X)\eP_\ell(Y)$  will  denote  a polynomial which is homogeneous of degree $k$ in the variables $X$ and of degree $\ell$ in the variables  $Y$.  We then have the equalities

\begin{equation}
\eH_a=\eP_{1}(\pa\eta)\eH_{a-1},
\label{eq: pa1}
\end{equation}
\begin{equation}
\pa^2\eH_a= \eP_{2}(\pa\eta)\eH_{a-2}+ \eP_{1}(\pa^2\eta)\eH_{a-1},
\label{eq: pa2}
\end{equation}
\begin{equation}
\pa^3\eH_a=\eP_{3}(\pa\eta)\eH_{a-3} + \eP_{1}(\pa\eta)\eP_1(\pa^2\eta)\eH_{a-2} +\eP_1(\pa^3\eta)\eH_{a-1},
\label{eq: pa3}
\end{equation}
\begin{equation}
\begin{split}
\pa^4\eH_a= \eP_4(\pa\eta)\eH_{a-4}+ \bigl(\,  \eP_2(\pa\eta)\eP_{1}(\pa^2\eta) \,\bigr)\eH_{a-3}\\
+\bigl(\, \eP_2(\pa^2\eta)+\eP_1(\pa\eta)\eP_1(\pa^3\eta)\,\bigr)\eH_{a-2}+\eP_1(\pa^4\eta)\eH_{a-1}.
\end{split}
\label{eq: pa4}
\end{equation}

To simplify the presentation we will assume that in (\ref{eq: two-step}) we have $i=1$, $j=2$. Also,   we will denote by $O(1)$  a function $f(x,y)$ such that $f(x,y)|_{x=y=0}=0$. The computations  in Section \ref{s: dist} show that  for $x=0$ terms of the form $\eP_j(\pa\eta)$ and $\eP_k(\pa^3\eta)$ are $O(1)$.   In particular, the above equalities show that the 1st and 3rd order derivatives of $\eH^a$ are $O(1)$.  We have
\begin{equation}
\begin{split}
\pa^2_{x^1}\pa^2_{y^2}(U_k \eH_k)= \pa^2_{x^1}\left(\,\bigl(\pa^2_{y^2}U_k\bigr) \eH_k + 2\pa_{y^2}U_k \pa_{y^2}\eH_k + U_k\pa^2_{y^2}\eH_k\right)\\
=\bigl(\pa^2_{x^1}\pa^2_{y^2}U_k \bigr)\eH_k +\bigl(\pa^2_{y^2}U_k\bigr) \bigl(\pa^2_{x^1}\eH_k)+\bigl(\pa^2_{x^1}U_k\bigr) \bigl(\pa^2_{y^2}\eH_k)\\
+4(\pa^2_{x^1y^2}U_k) (\pa^2_{x^1y^2}\eH_k) + U_k\pa^2_{x^1}\pa^2_{y^2}\eH_k +O(1),
\end{split}
\label{eq: curv1}
\end{equation}
\begin{equation}
\begin{split}
\pa^2_{x^1x^2}\pa^2_{y^1y^2}(U_k \eH_k)= \pa^2_{x^1x^2}\left( \,\bigl(\pa^2_{y^1y^2}U_k \bigr)\eH_k +\pa_{y^1}U_k\pa_{y^2}\eH_k+\pa_{y^2}U_k\pa_{y^1}\eH_k +U_k\pa^2_{y^1y^2}\eH_k\,\right)\\
= \bigl(\pa^2_{x^1x^2}\pa^2_{y^1y^2}U_k \bigr)\eH_k + \bigl(\pa^2_{y^1y^2}U_k \bigr) \bigl(\pa^2_{x^1x^2}\eH_k\bigr)\\
+\pa^2_{x^2y^1}U_k\pa^2_{x^1y^2}\eH_k +\pa^2_{x^1y^1}U_k\pa^2_{x^2y^2}\eH_k+\pa^2_{x^2y^2}U_k\pa_{x^1y^1}\eH_k +\pa^2_{x^1y^2}U_k\pa^2_{x^2y^1}\eH_k\\
+ \pa^2_{x^1x^2}U_k\pa^2_{y^1y^2}\eH_k+ U_k\pa^2_{x^1x^2}\pa^2_{y^1y^2}\eH_k +O(1)
\end{split}
\label{eq: curv2}
\end{equation}
He deduce that
\[
\Bigl(\pa^2_{x^1}\pa^2_{y^2}-\pa^2_{x^1x^2}\pa^2_{y^1y^2}\Bigr)(U_k \eH_k)_{x=y=0} = \sum_{j=0}^4T_k^j\eH^{k-j}|_{x=y=0},
\]
where the coefficients $T_k^j$ are polynomials in the derivatives of $U_k$ and $\eta$ at $(x,y)=(0,0)$. Using (\ref{eq: pa1})-(\ref{eq: pa4}) we deduce
\[
T_k^4=T_k^3=0.
\]
Moreover, the terms  that appear in $T_k^2$  appear only when we take forth order  derivatives of $\eH_k$.  Upon inspecting (\ref{eq: curv1}) and (\ref{eq: curv2}) we see that the 4th order derivatives of $\eH_k$ are multiplied by  $U_k$.   According to (\ref{eq: Ub})  the function $U_k$ is  $O(1)$ if $k>1$. Hence $T_k^2=0$ for $k>1$. We deduce
\[
K^{ii,jj}_t-K^{ij,ij}_t  \sim \sum_{k=1}^\infty d_m(2k)\left(T_k^0\eH_k+T_k^1\eH_{k-1}+ T_k^2\eH_{k-2}\right)|_{x=y=0}
\]
\[
=B_{-1}\eH_{-1}|_{x=y=0}+B_0\eH_0|_{x=y=0} +B_1\eH_1|_{x=y=0}+\cdots,
\]
where
\[
B_{-1} =d_m(2) T^2_1,\;\; B_0 =d_m(2) T_1^1,\;\;B_1=d_m(2) T^0_1+ d_m(4)T^1_2,\dotsc.
\]
The term $B_{-1}$ can be alternatively described as
\[
B_{-1}= A_{m,ii;jj,0}-A_{m,ij;ij,0},
\]
where the coefficients  $A_{m,\alpha,\beta,0}$ are defined as in (\ref{eq: wave-ker2}).  Using (\ref{eq: Am}) and  (\ref{eq: Am1}) we deduce
\[
B_{-1}=0.
\]
To compute $T_1^1$ we observe first that
\begin{equation}
\eta(x-y)=\sum_i(x^i-y^i)^2+\mbox{higer order terms}.
\label{eq: eta}
\end{equation}
Using (\ref{eq: pa2}) we can simplify (\ref{eq: curv1}) and (\ref{eq: curv2}) in the case $k=1$ as follows.
\begin{equation}
\pa^2_{x^1}\pa^2_{y^2}(U_1 \eH_1)=\bigl(\pa^2_{x^1}\pa^2_{y^2}U_1 \bigr)\eH_1 + U_1\pa^2_{x^1}\pa^2_{y^2}\eH_1 +O(1),
\label{eq: curv3}
\end{equation}
\begin{equation}
\begin{split}
\pa^2_{x^1x^2}\pa^2_{y^1y^2}(U_1 \eH_1)
= \bigl(\pa^2_{x^1x^2}\pa^2_{y^1y^2}U_1 \bigr)\eH_1 +\pa^2_{x^1y^1}U_1\pa^2_{x^2y^2}\eH_1\\
+\pa^2_{x^2y^2}U_1\pa_{x^1y^1}\eH_1
+ U_1\pa^2_{x^1x^2}\pa^2_{y^1y^2}\eH_1 +O(1).
\end{split}
\label{eq: curv4}
\end{equation}
Using (\ref{eq: pa2}), (\ref{eq: pa4}) and (\ref{eq: eta}) we deduce that
\[
\begin{split}
T_1^1=\left(\pa^2_{x^1}\pa^2_{y^2}-\pa^2_{x^1x^2}\pa^2_{y^1y^2}\right)\eta|_{(0,0)}\\
+2\left(\pa^2_{x^1}U_1+\pa^2_{y^2}U_1\right)|_{(0,0)}+ 2\left(\pa^2_{x^1y^1}U_1+\pa^2_{x^2y^2}U_1\right)|_{(0,0)}.
\end{split}
\]
Using the transport equation (\ref{eq: Ua}) we  obtain as in \cite[VI.3]{Cha} that $U_1$ coincides  with the function $\vfi(x,y)$ in \cite[VI.3 Eq.(33)]{Cha} or the function $u_0(x,y)$ in \cite[p. 380]{BBG}.  For our purposes an explicit description  of $U_1$ is not needed. All we care is that
\[
U_1(x,y)=U_1(y,x),\;\;U_1(x,x)\equiv 1.
\]
 These conditions imply that the Hessian of $U_1(x,y)$  at $(0,0)$ is a quadratic form in the variables $u_i=(x^i-y^i)$ so that
 \[
 \pa^2_{x^1}U_1(0,0)+\pa^2_{x^1y^1}U_1(0,0)=\pa^2_{y^2}U_1(0,0)+\pa^2_{x^2y^2}U_1(0,0)=0.
 \]
Hence
\[
T_1^1=\left(\pa^2_{x^1}\pa^2_{y^2}-\pa^2_{x^1x^2}\pa^2_{y^1y^2}\right)\eta|_{(0,0)}.
\]
Using (\ref{eq: r4})  we conclude that
\[
T_1^1= ZR_{1212} = Z K_{12}(\bp), 
\]
where $Z$ is a universal constant, independent of   $(M,g)$.  Hence
\[
K^{ii,jj}_t-K^{ij,ij}_t  \sim d_m(2)ZK_{12}(\bp) \eH^0|_{x=y=0}  +\sum_{k\geq 1}B_k\eH_k|_{x=y=0}.
\]
The   equality (\ref{eq: two-step})  now follows   from the above  equality by using  (\ref{eq: wave-ker4}), (\ref{eq: prope})  and Lemma \ref{lemma: abs}.
\end{proof}

\section{Gaussian measures and Gaussian vectors}
\label{s: gauss}
\setcounter{equation}{0}

For the reader's convenience we  survey here a few  basic facts about Gaussian  measures. For more details we refer to \cite{Bog}.  A   \emph{Gaussian measure} on $\bR$  is a Borel measure $\gamma_{\mu,v}$, $v\geq 0$, $m\in\bR$,   of the form
\[
\gamma_{\mu,v}(x)= \frac{1}{\sqrt{2\pi v}} e^{-\frac{ (x-\mu)^2}{2v}} dx.
\]
The scalar $\mu$ is called the \emph{mean}, while $v$ is called the \emph{variance}. We allow $v$ to be zero in which case
\[
\gamma_{\mu,0}=\delta_\mu=\mbox{the Dirac measure on $\bR$ concentrated at $\mu$}.
\]
 For a  real valued random variable  $X$ we write
 \begin{equation}
 X\in \bsN(\mu, v)
 \label{eq: normal}
 \end{equation}
 if the probability   measure of $X$ is $\gamma_{\mu, v}$.

 Suppose that $\bsV$ is a finite dimensional vector space. A \emph{Gaussian measure} on $\bsV$ is a  Borel measure  $\gamma$ on $\bsV$ such that, for any $\xi\in\bsV\dual$, the pushforward $\xi_*(\gamma)$ is a Gaussian measure on $\bR$, 
\[
\xi_*(\gamma)=\gamma_{\mu(\xi),v(\xi)}.
\]
One can show that  the map $\bsV\dual\ni \xi\mapsto \mu(\xi)\in\bR$ is linear, and  thus can be identified with a  vector $\bmu_\gamma\in \bsV$ called the \emph{barycenter} or \emph{expectation} of $\gamma$ that can be alternatively defined by the equality 
\[
\bmu_\gamma=\int_{\bsV} \bv d\gamma(\bv). 
\]
Moreover, there exists  a   nonnegative definite, symmetric bilinear map
\[
\bSi: \bsV\dual\times\bsV\dual\ra \bR\;\;\mbox{such that}\;\;v(\xi)= \bSi(\xi,\xi),\;\;\forall \xi\in \bsV\dual.
\]
The form  $\bSi$ is called the \emph{covariance form} and can be identified with a  linear operator $\bsS:\bsV\dual\ra \bsV$ such that
\[
\bSi(\xi,\eta)=\lan \xi, \bsS\eta\ran,\;\;\forall \xi,\eta\in \bsV\dual,
\]
where $\lan-,-\ran:\bsV\dual\times\bsV\ra \bR$ denotes the natural bilinear pairing between a vector space and its dual. The operator $\bsS$ is called the \emph{covariance operator} and it is explicitly described by the  integral formula
\[
\lan \xi, \bsS\eta\ran=\bSi(\xi,\eta)=\int_{\bsV}\lan \xi,\bv-\bmu_\gamma\ran \lan\eta, \bv-\bmu_\gamma\ran d\gamma(\bv).
\]
The Gaussian measure is said to be \emph{nondegenerate} if $\bSi$ is nondegenerate, and it is called \emph{centered} if $\bmu=0$. A Gaussian measure  on $\bsV$ is uniquely determined by  its covariance form and its  expectation. 

\begin{ex}  Suppose that   $\bsU$ is an $n$-dimensional  Euclidean space with  inner product $(-,-)$. We use the inner product to identify $\bsU$ with its dual $\bsU\dual$. If $A:\bsU\ra \bsU$ is a symmetric, positive definite   operator,  then
\begin{equation}
\bgamma_A(d\bu) =\frac{1}{(2\pi)^{\frac{n}{2}}\sqrt{\det A}} e^{-\frac{1}{2}(A^{-1}\bu, \bu)}\,|d\bu|
\label{eq: gamaA}
\end{equation}
is a  centered  Gaussian  measure on $\bsU$ with covariance form described by the  operator $A$.\qed
\end{ex}

If  $\bsV$ is a  finite dimensional vector space  equipped with a  Gaussian measure  $\gamma$  and $\bsL: \bsV\ra \bsU$ is a   linear  map, then the pushforward $\bsL_*\gamma$ is a  Gaussian measure on   $\bsU$ with expectation $\bmu_{\bsL_*\gamma}=\bsL(\bmu_\gamma)$ and covariance form
\[
\bSi_{\bsL_*\gamma}:\bsU\dual\times \bsU\dual\ra \bR,\;\; \bSi_{\bsL_*\gamma}(\eta,\eta)= \bSi_\gamma(\bsL\dual\eta,\bsL\dual\eta),\;\;\forall \eta \in \bsU\dual,
\]
where $\bsL\dual:\bsU\dual\ra \bsV\dual$ is the dual (transpose) of the linear map $\bsL$. Observe that if $\gamma$ is nondegenerate and $\bsL$ is surjective, then $\bsL_*\gamma$ is also nondegenerate.

Suppose $(\eS, \mu)$ is a probability space.    A \emph{Gaussian} random vector on $(\eS,\mu)$ is a (Borel) measurable map
\[
X: \eS\ra \bsV,\;\;\mbox{$\bsV$ finite dimensional vector space}
\]
such that $X_*\mu$ is a Gaussian measure on $\bsV$. We will refer to this measure as the \emph{associated Gaussian measure}, we denote it  by $\gamma_X$ and  we denote by $\bSi_X$ (respectively $\bsS(X)$) its covariance form (respectively operator),
\[
\bSi_X(\xi_1,\xi_2)=\bsE\bigl(\, \lan \xi_1, X-\bsE(X)\,\ran\,\lan \xi_2, X-\bsE(X)\,\ran\,\bigr).
\]
 Note that the   expectation of $\gamma_X$ is precisely the expectation of $X$. The random vector is called \emph{nondegenerate}, respectively \emph{centered}, if the Gaussian measure  $\gamma_X$ is such.
 
 Let us point out that if $X:\eS\ra \bsU$ is a Gaussian random vector and $\bsL:\bsU\ra \bsV$ is a linear map, then the random vector $\bsL X:\eS\ra \bsV$ is also Gaussian. Moreover
\[
\bsE(\bsL X)= \bsL \bsE(X),\;\;\bSi_{\bsL X}(\xi,\xi)=\bSi_X(\bsL\dual\xi,\bsL\dual\xi),\;\;\forall \xi\in\bsV\dual,
\]
where $\bsL\dual: \bsV\dual\ra \bsU\dual$ is the linear map dual to $\bsL$. Equivalently,  $\bsS(\bsL X)= \bsL \bsS(X )\bsL\dual$.

 Suppose  that $X_j:\eS\ra \bsV_1$, $j=1,2$, are two \emph{centered} Gaussian random vectors such that the direct sum $X_1\oplus X_2:\eS\ra \bsV_1\oplus \bsV_2
$ is also a centered  Gaussian random vector with  associated  Gaussian measure
\[
\gamma_{X_1\oplus X_2}= p_{X_1\oplus X_2} (\bx_1,\bx_2) |d\bx_1d\bx_2|.
\]
 We obtain a bilinear form
\[
\cov(X_1,X_2):\bsV_1\dual\times \bsV_2\dual\ra \bR,\;\; \cov(X_1,X_2)(\xi_1,\xi_2)=\bSi(\xi_1,\xi_2),
\]
called the \emph{covariance form}. The   random vectors $X_1$ and $X_2$ are independent if and only if  they are uncorrelated, i.e.,
\[
\cov(X_1,X_2)=0.
\]
We can then identify  $\cov(X_1,X_2)$ with a linear operator $\Cov(X_1,X_2):\bsV_2\ra \bsV_1$, via the equality
\[
\begin{split}
\bsE\bigl(\,\lan \xi_1,X_1\ran\lan\xi_2, X_2\ran\,\bigr) &=\cov(X_1,X_2)(\xi_1,\xi_2)\\
&=\bigl\lan\, \xi_1, \Cov(X_1,X_2)\xi_2^\dag\,\bigr\ran,\;\;\forall \xi_1\in\bsV_1\dual,\;\;\xi_2\in\bsV_2\dual,
\end{split}
\]
where $\xi_2^\dag\in\bsV_2$ denotes the vector metric dual to $\xi_2$. The  operator $\Cov(X_1,X_2)$ is called the \emph{covariance operator} of $X_1, X_2$.  

The  conditional random variable $(X_1|X_2=x_2)$  has probability density 
\[
p_{(X_1|X_2=\bx_2)}(\bx_1)=  \frac{p_{X_1\oplus X_2}(\bx_1,\bx_2)}{\int_{\bsV_1} p_{X_1\oplus X_2}(\bx_1,\bx_2) |d\bx_1|}.
\]
For a measurable function $f:\bsV_1\ra \bR$ the conditional expectation   $\bsE(f(X_1)|X_2=\bx_2)$ is the (deterministic) scalar
\[
\bsE(f(X_1)|X_2=\bx_2)= \int_{\bsV_1} f(\bx_1) p_{(X_1|X_2=\bx_2)}(\bx_1)|d\bx_1|.
\]
If $X_2$ is nondegenerate,   the  \emph{regression formula}, \cite{AzWs}, implies that  the random vector   $(X_1|X_2=x_2)$ is a  Gausian vector with covariance operator
\begin{equation}
\bsS(Y)=\bsS(X_1)-\Cov(X_1,X_2)\bsS(X_2)^{-1} \Cov(X_2,X_1),
\label{eq: cov-regr}
\end{equation}
and   mean  
\begin{equation}
\bsE(X_1|X_2=x_2)=Cx_2,
\label{eq: cond-expect}
\end{equation}
 where  $C$ is given by
\begin{equation}
 C= \Cov(X_1,X_2)\bsS(X_2)^{-1}.
\label{eq: gauss-c}
\end{equation}

\section{A class of random symmetric matrices}
\label{s: gmat}
\setcounter{equation}{0}

We denote by $\Sym_m$ the space of real symmetric   $m\times m$  matrices. This is an Euclidean space with respect to the inner product
\[
(A,B):=\tr(AB).
\]
This inner product is invariant with respect to the action of $\SO(m)$ on $\Sym_m$. We set
\[
\widehat{\bsE}_{ij}:=\begin{cases}
\bsE_{ij}, & i=j\\
\frac{1}{\sqrt{2}}E_{ij}, & i<j.
\end{cases}.
\]
The collection  $(\widehat{\bsE}_{ij})_{i\leq j}$ is a basis  of $\Sym_m$ orthonormal with respect to the above inner product.  We set
\[
\hat{a}_{ij}:= \begin{cases}
a_{ij}, & i=j\\
\sqrt{2}a_{ij}, & i<j.
\end{cases}
\]
The collection $(\hat{a}_{ij})_{i\leq j}$  the orthonormal basis of $\Sym_m\dual$ dual to $(\widehat{\bsE}_{ij})$.  The volume density induced by this metric is
\[
|dA|:=\prod_{i\leq j} d\widehat{a}_{ij}= 2^{\frac{1}{2}\binom{m}{2}}\prod_{i\leq j} da_{ij}.
\]
Throughout  the paper  we     encountered a $2$-parameter family of Gaussian    probability measures  on $\Sym_m$.  More precisely for any    real numbers  $u,v$ such that
\[
v>0, mu+2v>0,
\]
 we  denote by $\Sym_m^{u,v}$  the space  $\Sym_m$     equipped with the centered Gaussian measure $d\bGamma_{u,v}(A)$ uniquely determined by the covariance equalities
 \[
 \bsE(a_{ij}a_{k\ell})=  u\delta_{ij}\delta_{k\ell}+ v(\delta_{ik}\delta_{j\ell}+ \delta_{i\ell}\delta_{jk}),\;\;\forall 1\leq i,j,.k,\ell\leq m.
 \]
 In particular we have
 \[
 \bsE(a_{ii}^2)= u+2v,\;\;\bsE(a_{ii}a_{jj})=u,\;\;\;\bsE(a_{ij}^2)=v,\;\;\forall 1\leq i\neq j\leq m,
 \]
while all other covariances are trivial.  The  ensemble  $\Sym_m^{0,v}$ is    a rescaled version of of the Gaussian Orthogonal Ensemble  (GOE) and we will refer to it as $\GOE_m^v$.       

For $u>0$ the ensemble $\Sym_m^{u,v}$ can be given an alternate description.   More precisely   a random $A\in \Sym_m^{u,v}$ can be described as a sum
\[
A= B+ \ X\one_m,\;\;B\in \GOE_m^v,\;\; X\in \bsN(0, u),\;\;\mbox{ $B$ and $X$ independent}.
\]
We write  this
\begin{equation}
\Sym_m^{u,v} =\GOE_m^v\hat{+}\bsN(0,u)\one_m,
\label{eq: smgoe}
\end{equation}
where $\hat{+}$ indicates a sum of \emph{independent} variables.

The  Gaussian measure $d\bGamma_{u,v}$ coincides with the Gaussian measure $d\bGamma_{u+2v,u,v}$ defined in \cite[App. B]{N2}.  We recall a few facts from   \cite[App. B]{N2}. 

The  probability density  $d\bGamma_{u,v}$  has the explicit description
\[
d\bGamma_{u,v}(A)= \frac{1}{(2\pi)^{\frac{m(m+1)}{4}}  \sqrt{D(u,v)}} e^{-\frac{1}{4v}\tr A^2-\frac{u'}{2}(\tr A)^2} |dA|,
\]
 where
 \[
 D(u,v)= (2v)^{(m-1)+\binom{m}{2}}\bigl( mu +2v\,\bigr),
 \]
 and 
 \[
 u'=\frac{1}{m}\left(\frac{1}{mu+2v}-\frac{1}{2v}\right)=-\frac{u}{2v(mu+2v)}.
  \]
 In the special case  $\GOE_m^v$ we have $u=u'=0$  and  
\begin{equation}
d\bGamma_{0,v}(A)=\frac{1}{(2\pi v)^{\frac{m(m+1)}{4}} } e^{-\frac{1}{4v}\tr A^2} |dA|.
\label{eq: gov}
\end{equation}
We have a   \emph{Weyl integration formula} \cite {AGZ} which states that if  $f: \Sym_m\ra \bR$ is a measurable  function which  is invariant under  conjugation, then the  the   value $f(A)$ at $A\in\Sym_m$ depends only on the eigenvalues $\lambda_1(A)\leq \cdots \leq \lambda_n(A)$ of $A$ and we  have
\begin{equation}
\bsE_{\GOE_m^v}\bigl(\,f(X)\,\bigr)=\frac{1}{\bsZ_m(v)} \int_{\bR^m}  f(\lambda_1,\dotsc ,\lambda_m) \underbrace{\left(\prod_{1\leq i< j\leq m}|\lambda_i-\lambda_j| \right)\prod_{i=1}^m e^{-\frac{\lambda_i^2}{4v}} }_{=:Q_{m,v}(\lambda)} |d\lambda_1\cdots d\lambda_m|,
\label{eq: weyl}
\end{equation}
where the normalization constant $\bsZ_m(v)$ is defined  by
\[
{\bsZ_m(v)} =\int_{\bR^m}   \prod_{1\leq i< j\leq m}|\lambda_i-\lambda_j| \prod_{i=1}^m e^{-\frac{\lambda_i^2}{4v}} |d\lambda_1\cdots d\lambda_m|
\]
\[
=(2v)^{\frac{m(m+1)}{4}} \underbrace{\int_{\bR^m}   \prod_{1\leq i< j\leq m}|\lambda_i-\lambda_j| \prod_{i=1}^m e^{-\frac{\lambda_i^2}{2}} |d\lambda_1\cdots d\lambda_m|}_{=:\bsZ_m}.
\]
The precise value of $\bsZ_m$  can be computed via Selberg integrals, \cite[Eq. (2.5.11)]{AGZ}, and we have
\begin{equation}
\bsZ_m=(2\pi)^{\frac{m}{2}} m!\prod_{j=1}^{m}\frac{\Gamma(\frac{j}{2})}{\Gamma(\frac{1}{2})}=2^{\frac{m}{2}}m!\prod_{j=1}^m \Gamma\left(\frac{j}{2}\right).
\label{eq: zm}
\end{equation}
For any positive integer $n$ we define the \emph{normalized} $1$-point corelation function $\rho_{n,v}(x)$ of $\GOE_n^v$ to be
\[
\rho_{n,v}(x)= \frac{1}{\bsZ_n(v)}\int_{\bR^{n-1}} Q_{n,v}(x,\lambda_2,\dotsc, \lambda_n) d\lambda_1\cdots d\lambda_n.
\]
For any Borel measurable function $f:\bR\ra \bR$  we have \cite[\S 4.4]{DG}
 \begin{equation}
 \frac{1}{n}\bsE_{\GOE_n^v} \bigl(\,\tr f(X)\,\bigr)= \int_{\bR} f(\lambda) \rho_{n,v}(\lambda) d\lambda. 
 \label{eq: 1pcor}
 \end{equation}
The equality (\ref{eq: 1pcor}) characterizes $\rho_{n,v}$. Let us  observe that for any constant $c>0$, if 
\[
A\in \GOE_n^v\Llra cA\in \GOE_n^{c^2v}.
\]
Hence for any Borel set $B\subset \bR$   we have
\[
\int_{cB} \rho_{n,c^2v}(x) dx =\int_B \rho_{n,v} (y) dy.
\]
 We conclude that
 \begin{equation}
 c\rho_{n,c^2v}(cy)= \rho_{n,v}(y),\;\;\forall n,c,y.
 \label{eq: resc-cor}
 \end{equation}
 The   behavior of the $1$-point correlation  function $\rho_{n,v}(x)$ for $n$ large  is described by \emph{Wigner semicircle theorem}  \cite[Thm.2.1.1]{AGZ} which states that for any $v>0$ the sequence  of probability measures on $\bR$
\[
\rho_{n,vn^{-1}}(x) dx =n^{\frac{1}{2}}\rho_{n,v}(n^{\frac{1}{2}} x) dx
\]
converges weakly  as $n\ra \infty$ to the  semicircle distribution
\[
\rho_{\infty,v}(x)|dx|= \bsI_{\{|x|\leq 2\sqrt{v}\}}\frac{1}{2\pi v}\sqrt{4v-x^2} |dx|.
\]
The   expected value of the absolute value of the  determinant of     of a  random $A\in \GOE_m^v$ can be expressed neatly in terms of the correlation function $\rho_{m+1,v}$.   More precisely, we have the following result first observed by Y.V. Fyodorov \cite{Fy} in a context related to ours.

\begin{lemma}  Suppose $v>0$. Then for any $c\in\bR$ we have
\[
\bsE_{\GOE_m^v} \bigl(\,|\det(A-c\one_m)|\,\bigr)= 2^{\frac{3}{2}}(2v)^{\frac{m+1}{2}}\Gamma\left(\frac{m+3}{2}\right) e^{\frac{c^2}{4v}}\rho_{m+1,v}(c) .
\]
\label{lemma: exp-det}
\end{lemma}

\begin{proof}Using the Weyl integration formula we deduce
\[
\bsE_{\GOE_m^v} \bigl(\,|\det(A-c\one_m)|\,\bigr) =\frac{1}{\bsZ_m(v)}\int_{\bR^m} \prod_{i=1}^me^{-\frac{\lambda_i^2}{4v}}|c-\lambda_i|\prod_{i\leq j}|\lambda_i-\lambda_j| d\lambda_1\cdots d\lambda_m
\]
\[
=\frac{e^{\frac{c^2}{4v}}}{\bsZ_m(v)} \int_{\bR^m} e^{-\frac{c^2}{4v}}\prod_{i=1}^me^{-\frac{\lambda_i^2}{4v}}|c-\lambda_i|\prod_{i\leq j}|\lambda_i-\lambda_j| d\lambda_1\cdots d\lambda_m
\]
\[
= \frac{e^{\frac{c^2}{4v}}\bsZ_{m+1}(v)}{\bsZ_m(v)}\frac{1}{\bsZ_{m+1}(v)} \int_{\bR^m} Q_{m+1,v}(c,\lambda_1,\dotsc, \lambda_m) d\lambda_1\cdots d\lambda_m
\]
\[
=\frac{e^{\frac{c^2}{4v}}\bsZ_{m+1}(v)}{\bsZ_m(v)} \rho_{m+1,v}(c)=v^{\frac{m+1}{2}}\frac{e^{\frac{c^2}{4v}}\bsZ_{m+1}}{\bsZ_m} \rho_{m+1,v}(c)
\]
\[
=(m+1)\sqrt{2}(2v)^{\frac{m+1}{2}}e^{\frac{c^2}{4v}}\Gamma\left(\frac{m+1}{2}\right) \rho_{m+1,v}(c)= 2^{\frac{3}{2}}(2v)^{\frac{m+1}{2}}\Gamma\left(\frac{m+3}{2}\right) e^{\frac{c^2}{4v}} \rho_{m+1,v}(c) .
\]
\end{proof}

The above result admits the following generalization, \cite[Lemma 3.2.3]{Auff0}.

\begin{lemma} Let $u,v>0$.  Set
\[
\theta^+_{m,v}(x):= \rho_{m+1,v}(x)e^{\frac{x^2}{4v}}.
\]
Then
\begin{subequations}
\begin{equation}
\bsE_{\Sym_m^{u,v}}\bigl(\,|\det (A-c\one_m)|\,\bigr)= 2^{\frac{3}{2}}(2v)^{\frac{m+1}{2}}\Gamma\left(\frac{m+3}{2}\right) \frac{1}{\sqrt{2\pi u}}\int_{\bR}\rho_{m+1,v}(c-x) e^{\frac{(c-x)^2}{4v}-\frac{x^2}{2u}} dx 
\label{eq: epec-uv0}
\end{equation}
\begin{equation}
=2^{\frac{3}{2}}(2v)^{\frac{m+1}{2}}\Gamma\left(\frac{m+3}{2}\right) (\gamma_u\ast\theta^+_{m+1,v}) (c).
\label{eq: expec_uv}
\end{equation}
\end{subequations}
In particular, if  $u=2kv$, $k<1$ we have
\[
\bsE_{\Sym_m^{2kv,v}}\bigl(\,|\det (A-c\one_m)|\,\bigr)= 2^{\frac{3}{2}}(2v)^{\frac{m}{2}}\Gamma\left(\frac{m+3}{2}\right) \frac{1}{\sqrt{2\pi k}}\int_{\bR}\rho_{m+1,v}(c-x) e^{ -\frac{1}{4vt_k^2}(x + t_k^2c)^2 +\frac{(t_k^2+1)c^2}{4v}} dx,
\]
($\lambda:=c-x$)
\[
= 2^{\frac{3}{2}}(2v)^{\frac{m}{2}}\Gamma\left(\frac{m+3}{2}\right) \frac{1}{\sqrt{2\pi k}}\int_{\bR}\rho_{m+1,v}(\lambda) e^{ -\frac{1}{4vt_k^2}(\lambda -(t_k^2+1)c)^2 +\frac{(t_k^2-1)c^2}{4v}} d\lambda
\]
where
\[
t_k^2:=\frac{1}{\frac{1}{k}-1}=\frac{k}{1-k}.
\]
\label{lemma: exp-det1}
\end{lemma}

\begin{proof}   Recall the equality (\ref{eq: smgoe}) 
\[
\Sym_m^{u,v} =\GOE_m^v\hat{+}\bsN(0,u)\one_m.
\]
We deduce that
\[
\bsE_{\Sym_m^{u,v}}\bigl(\,|\det (A-c\one_m)|\,\bigr)=\bsE\bigl(\,\det(B+(X-c)\one)|\,\bigr)
\]
\[
=\frac{1}{\sqrt{2\pi u}}\int_{\bR}\bsE_{\GOE_m^v}\bigl(\, |\det (B-(c-X)\one_m)|\;\bigl|\; X=x)  e^{-\frac{x^2}{2u}} dx
\]
\[
=\frac{1}{\sqrt{2\pi u}}\int_{\bR}\bsE_{\GOE_m^v}\bigl(\, |\det (B-(c-x)\one_m)|\,\bigr)  e^{-\frac{x^2}{2u}} dx
\]
\[
= 2^{\frac{3}{2}}(2v)^{\frac{m+1}{2}}\Gamma\left(\frac{m+3}{2}\right) \frac{1}{\sqrt{2\pi u}}\int_{\bR}\rho_{m+1,v}(c-x) e^{\frac{(c-x)^2}{4v}-\frac{x^2}{2u}} dx.
\]
Now observe that if $ u=2k v$ then
\[
\frac{(c-x)^2}{4v}-\frac{x^2}{2u}=  -\frac{x^2}{4k v}+\frac{1}{4v}(x^2 -2cx+ c^2)
\]
\[
=\frac{1}{4v} \left( -\frac{1}{t_k^2}x^2-2cx -c^2t_k^2\right) + \frac{c^2(1+t_k^2)}{4v}= -\frac{1}{4vt_k^2}(x+t_k^2c)^2+ \frac{c^2(1+t_k^2)}{4v}.
\]

\end{proof}

\end{document}